\documentclass[11pt]{article}

\usepackage[margin=1in]{geometry}  
\usepackage{graphicx}           
\usepackage{amsmath, accents}              
\usepackage{amsfonts}            
\usepackage{dsfont} 
\usepackage{mathrsfs}
\usepackage{amsthm}    
\usepackage{amssymb}           
\usepackage{soul}
\usepackage{xcolor}
\usepackage{longtable}
\usepackage{float}
\usepackage{algorithm}
\usepackage{algpseudocode}

\newtheorem{thm}{Theorem}[section]

\newtheorem{defi}[thm]{Definition}
\newtheorem{example}[thm]{Example}
\newtheorem{remark}[thm]{Remark}
\newtheorem{prop}[thm]{Proposition}

\theoremstyle{remark}

\newcommand{\Real}{\mathbb{R}}      

\begin{document}

\nocite{*}

\title{A Multigrid Optimization Algorithm for the Numerical Solution of Quasilinear Variational Inequalities Involving the $p$-Laplacian\thanks{Supported in part by the Escuela Polit\'ecnica Nacional del Ecuador, under the project PIMI 14-12 ``Viscoplastic Fluids in Food Industry'' and the MATH-AmSud Project ``SOCDE-Sparse Optimal Control of Differential Equations: Algorithms and Applications''. This paper was developed within the Master Program in Optimization of the Mathematics Department at Escuela Polit\'ecnica Nacional del Ecuador.}}

\author{Sergio~Gonz\'alez-Andrade and Sof\'ia L\'opez\\\small Research Center on Mathematical Modeling (MODEMAT) and \\\small Departamento de Matem\'atica -  Escuela Polit\'ecnica Nacional\\\small Ladr\'on de
Guevara E11-253, Quito 170413, Ecuador\\\small \tt sergio.gonzalez@epn.edu.ec and sofia.lopezo@epn.edu.ec} 
\date{\today}
\maketitle

\begin{abstract}
In this paper we propose a multigrid optimization algorithm (MG/OPT) for the numerical solution of a class of quasilinear variational inequalities of the second kind. This approach is enabled by  the fact that the solution of the variational inequality is given by the minimizer of a nonsmooth energy functional, involving the $p$-Laplace operator. We propose a Huber regularization of the functional and a finite element discretization for the problem. Further, we analyze the regularity of the discretized energy functional, and we are able to prove that its Jacobian is slantly differentiable. This regularity property is useful to analyze the convergence of the MG/OPT algorithm. In fact, we demostrate that the algorithm is globally convergent by using a mean value theorem for semismooth functions. Finally, we apply the MG/OPT algorithm to the numerical simulation of the viscoplastic flow of Bingham, Casson and Herschel-Bulkley fluids in a pipe. Several experiments are carried out to show the efficiency of the proposed algorithm when solving this kind of fluid mechanics problems.
  \vskip .2in

\vspace{0.2cm}
\end{abstract}

\section{Introduction}
In this paper, we focus on the development of a multigrid algorithm for the fast finite element solution of a class of quasilinear variational inequalities of the second kind. The main idea is the application of an efficient multigrid approach to this kind of problems which, typically, leads us to the solution of large systems.

Let $\Omega$ be an open and bounded set in $\Real^n$ with Lipschitz boundary $\partial \Omega$. We are concerned with the numerical solution of the following class of quasilinear variational inequalities of the second kind: find $u \in W_0^{1,p}(\Omega)$ such that 
\begin{equation*}
\int_{\Omega} | \nabla u|^{p-2} ( \nabla u,\nabla (v-u) )\,dx +g \int_{\Omega} | \nabla v |\,dx-g \int_{\Omega} | \nabla u |\,dx\geq \int_\Omega f (v-u)\,dx, \, \forall v \in W_0^{1,p}(\Omega),
\end{equation*}
where $1 < p < \infty$, $g > 0$ and $f \in L^q(\Omega).$ Here, $q=\frac{p}{p-1}$ stands for the conjugate exponent of $p$. It is known that these variational inequalities correspond to a first order necessary optimality condition for the following class of nonsmooth optimization problems.
\begin{equation}\label{eq:prob}
\underset{u\in W_0^{1,p}(\Omega)}{\min} J(u):= \frac{1}{p}\int_\Omega |\nabla u|^p\,dx + g\int_\Omega |\nabla u|\, dx - \int_\Omega f u\,dx.
\end{equation}
Consequently, we focus on the fast solution of this optimization problem. The existence and uniqueness of solutions for this problem has been analyzed and verified in previous contributions, such as \cite{Gonzalez1}. 

The variational inequalities under study provide a versatile tool in the study of a class of free boundary problems which arise in the modelling of complex fluids and materials. In fact, diverse problems including the flow of viscoplastic materials, the flow of electro- and magneto-rheological fluids and phenomena in glaciology have been successfully simulated by this kind of models (\cite{Gonzalez1,Ruzicka,Georg1}). 

Several approaches have been proposed for the numerical solution of problems like \eqref{eq:prob}. In \cite{Huilgol} an Augmented Lagrangian method is implemented for the numerical simulation of the flow of viscoplastic materials. In \cite{Gonzalez1}, a preconditioned descent algorithm is proposed and analyzed both in finite and infinite dimension settings. Regarding the use of multigrid algorithms, in \cite{korn1,korn2} the author proposes algorithms for variational inequalities of the first and the second kind, based on extended relaxation methods. In \cite{korn3}, the author proposes a multigrid algorithm for variational inequalities of the second kind using a combination of convex minimization with constrained Newton linearization. However, these contributions focus on variational inequalities involving linear elliptic operators such as the Laplacian.

The multigrid approach is a very appealing way to develop fast solution algorithms for the numerical approximation of \eqref{eq:prob}. In fact, the numerical solution of this kind of problems usually involves the resolution of large linear and nonlinear systems. Since these systems are computationally expensive to solve, the multigrid algorithms provide an efficient way to handle the large systems generated when discretizing the problem. Furthermore, it is natural to look for an algorithm which, in the context of the multigrid approximation, focus on the direct optimization of the energy functional. 

The multigrid optimization method (MG/OPT) corresponds to a nonlinear programming adaptation of the \textit{full approximation storage} (FAS) scheme. This approach is proposed, for instance in \cite{Lewis_Nash,Nash,BorziVallejo}, as an effective tool for large scale optimization problems. This algorithm works with different discretization levels of the optimization problem and takes advantage of the coarse problems to generate search directions for the finer problems. Similar approaches have been used for problems involving quasilinear operators, such as the $p$-Laplace operator (see \cite{Bermejo} and the references therein), but to the best of our knowledge, there are no contributions proposing a MG/OPT algorithm for variational inequalities of the second kind involving this kind of operators.

In this paper, we propose and analyze an MG/OPT algorithm to compute the finite element solution of a Huber regularized version of \eqref{eq:prob}. Considering the structure of the optimization problem, specifically the low regularity of the functional, we use a class of preconditioned descent algorithms proposed in \cite{Gonzalez1} as underlying optimization methods or smoothers. Further, the low regularity of the functional prevents us from doing a classical analysis of convergence. Therefore, we perform the convergence analysis of the MG/OPT algorithm by using a mean value theorem for \textit{Bouligand} differentiable functions, which is also applicable for semismooth functions. Finally, we present a comprehensive numerical experimentation focused on the numerical simulation of viscoplastic materials. Specifically, we focus on the flow of these materials through the cross-section of a pipe. 

Let us mention that, although the method developed in this article is concerned with variational inequalities of the second kind involving the nonsmooth term $\int_\Omega|\nabla u|\,dx$, the results can be extended to other variational inequalities of the second kind.

The paper is organized as follows: In section \ref{sec:prelim}, we present several results on generalized differentiability, which will be used to analyze the convergence of the multigrid method.  Since the problem is nonsmooth, in section \ref{sec:reganddisc} we propose a local regularization for the objective functional. Further, we present the finite element discretization of the problem. The MG/OPT method is presented in section \ref{sec:MGOPT}, whereas the convergence of the algorithm is discussed in section \ref{sec:convergence}. In section \ref{sec:impl}, a brief explanation of the underlying optimization and line search algorithms is presented. In section \ref{sec:flow}, we analyze the behaviour of the proposed methodology when applied to the numerical simulation of viscoplastic flow. We perform several experiments in order to show the main features of the algorithm. Finally, in section \ref{sec:conclusions}, we outline conclusions on this work and discuss future contributions.

\section{Preliminaries on Generalized Differentiability}\label{sec:prelim}
This section is devoted to the discussion of several concepts on generalized differentiability. We introduce the \textit{Bouligand} and the slant derivative of a nonsmooth function, and we discuss the relationship between these two concepts. Further, we present a mean value theorem for \textit{Bouligand} differentiable functions which also holds for semismooth functions. 
 
\begin{defi}
Let $X$ and $Y$ be two normed spaces, $D$ be a nonempty open set in $X$ and $J: D \subset X \rightarrow Y$ be a given mapping. For $x \in D$ and $h \in X$, if the limit
\begin{equation*}
J'(x)(h):=\lim_{t \rightarrow 0^{+}} \frac{J(x+th)-J(x)}{t}
\end{equation*}
exists, the function is said to be directionally differentiable. Further, $J'(x)(h)$ is the directional derivative of $J$ at $x$ in the direction $h$.
\end{defi}
\begin{remark}
From here on, by making a small notation abuse, we denote by $F '(u)$ the Fr\'echet derivative of $F$ at $u$, and by  $F'(u)(v)$ the directional derivative of $F$ at $u$ in the direction $v$.
\end{remark}

Next, we define the concept of \textit{Bouligand} differentiability and its relation to the semismoothnes concept. For further details, we refer the reader to \cite[Ch. 2, Sec. 2.1]{Ulbrich}.
\begin{defi}
 Let $D \subset \mathbb{R}^{n}$ be open and $J:D \rightarrow \mathbb{R}^{m}$ be Lipschitz continuos near $x \in D$, i.e., locally Lipschitz continuous. The set
\[
\partial_{B} J(x)=\left\lbrace M \in \mathbb{R}^{m \times n}: \exists (x_k) \subset D_J : x_k \rightarrow x, J'(x_k) \rightarrow M \right\rbrace
\]
is called Bouligand-subdifferential (or B-subdifferential) of $J$ at $x$. Here, $D_J \subset D$ is the set of all $x \in D$ at which $J$ admits a Fr\'echet derivative $J'(x) \in \mathbb{R}^{m \times n}$.
The convex hull of the Bouligand-subdifferential $\partial J(x)= co (\partial_{B} J(x))$ is the Clarke's generalized Jacobian of $J$ at $x$.
\end{defi}
\begin{defi}
Let  $J:D \rightarrow \mathbb{R}^{m}$ be defined on the open set $D \subset \mathbb{R}^{n}$. 
$J$ is called B-differentiable at $x \in D$ if $J$ is directionally differentiable at $x$ and
\[
\|J(x+h)-J(x)-J'(x)(h)\|= o(\|h\|) \hspace{5mm } \text{as} \hspace{5mm} h \rightarrow 0.
\] 
\end{defi}

\begin{prop}\label{def_bou_llc}
Let $D \subset \mathbb{R}^{n}$ be open and $J:D \rightarrow \mathbb{R}^{m}$ be a locally Lipschitz continuous function which is directionally differentiable at $x_0 \in D$. Then, the function $J$ is B-differentiable at $x_0$.
\end{prop}
\begin{proof}
\cite[Th.3.1.2 ]{Scholtes}
\end{proof}

Let us notice that the B-derivative of a locally Lipschitz continuous function is its directional derivative.

\begin{prop}\label{b-prop}
Let $J: D \rightarrow \mathbb{R}^{m}$ be defined on the open set $D \subset \mathbb{R}^{n}$ and let $J$ be B-differentiable at x. Then, $J'(x)(\cdot)$ is Lipschitz continuous, and, for every $h \in \mathbb{R}^{n}$, there exists $M \in \partial J(x)$ such that
\[
J'(x)(h)=M h.
\]
\end{prop}
\begin{proof}
 \cite[Sec. 8.2.1]{Kunisch-ito} 
\end{proof}

\begin{prop}\label{ssm_prop}
Let $J: D \rightarrow \mathbb{R}^{m}$ be defined on the open set $D \subset \mathbb{R}^{n}$. Then, for $x \in D$ the following statements are equivalent:
\begin{enumerate}
\item $J$ is semismooth at $x$.
\item $J$ is Lipschitz continuous near $x$, $J'(x)( \cdot )$ exists, and
\[
\displaystyle \sup_{M \in \partial J(x+h)} \|M h - J'(x)(h)\| = o(\|h\|) \hspace{5mm} \text{as} \hspace{5mm} h\rightarrow 0.
\]
\item $J$ is Lipschitz continuous near $x$, $J'(x)(\cdot)$ exists, and
\[
\displaystyle \sup_{M \in \partial J(x+h)} \|J(x+h) - J(x) - M h\| = o(\|h\|) \hspace{5mm} \text{as} \hspace{5mm} h\rightarrow 0.
\]
\end{enumerate}
\end{prop}
\begin{proof}
\cite[Prop.2.7]{Ulbrich}
\end{proof}

\begin{remark}\label{ssm_is_boul}
Let us notice that, from Proposition \ref{def_bou_llc} we have that $J$ is B-differentiable if it is Lipschitz continuous near $x$ (locally Lispchitz continuous at $x$) and directionally differentiable at $x$. Then, from Proposition \ref{ssm_prop}, we have that if $J$ is semismooth, $J$ is B-differentiable.
\end{remark}

\begin{prop}\label{b-prop2}
Let $J: D \rightarrow \mathbb{R}^{m}$ be defined on the open set $D \subset \mathbb{R}^{n}$ and let $J$ be semismooth in $x$. Then
\[
\| J'(x+h)(h) - J'(x)(h) \|= o(\|h\|) \hspace{5mm}\text{ as }  \hspace{5mm} h \rightarrow 0.
\]
\end{prop}
\begin{proof}
\cite[Th. 8.2]{Kunisch-ito} 
\end{proof}

Next, we introduce the notion of slant differentiability \cite[Sec.1]{Hintermuller}  that will be used in this work.

\begin{defi}\label{slanting-diff}
Let $X$ and $Y$ be Banach spaces, and $D$ be an open subset of $X$. A function $J: D \subset X \rightarrow Y$ is said to be slantly differentiable in $ D$, if there exists a family of mappings $J^{\circ}: D \rightarrow \mathcal{L}(X,Y)$ such that 
\begin{equation*}
\lim_{h \rightarrow 0} \frac{ \|J(x+h) - J(x) - J^{\circ}(x+h)h \|}{\|h\|}=0.
\end{equation*}
for every $x \in D$. The function $J^{\circ}$ is called a slanting function for $J$ in $D$.
\end{defi}

The previous definition was introduced in \cite{Hintermuller} as an adaptation of the definition of slant differentiability in Banch spaces stated in \cite{Chen-etal}, where the family of linear operators $\{J^{\circ}(x+h)\}$ is required to be uniformly bounded in the operator norm. Also, in \cite[pp. 868]{Hintermuller} the authors state that the notion of slant differentiability is a more general concept than the definition of semismoothness. In fact, the slanting functions are not required to be elements of the Clarke's generalized Jacobian $\partial J(x+h)$. However, a single-valued selection $M(x+h) \in \partial J(x+h)$, with $x \in D$, is a slanting function $J^{\circ}$- in the sense of Definition \ref{slanting-diff} - if Proposition \ref{ssm_prop} (item 3) holds for $x \in D$ (see \cite[Sec.1]{Hintermuller} for further details).

\begin{prop}\label{remark1}
If $J$ is continuously differentiable in a neighbourhood of $x$, then
\[
\partial J(x)=\partial_{B} J(x)= \{J'(x)\}
\]
\end{prop}
\begin{proof}
\cite[Prop.2.2]{Ulbrich}
\end{proof}

We now present an important example of a semismooth function that will be useful in the subsequent sections.
\begin{example}\label{ex:max} Let $g>0$ be a constant.  The mapping 
\begin{equation*}
\vec{z} \rightarrow \max(g,\gamma |\vec{z}|)
\end{equation*}
from $\mathbb{R}^n$ to $\mathbb{R}$ is semismooth on $\mathbb{R}^n$. Further, the slant derivative of this function is the characteristic function $\chi_{A_\gamma}(\vec{z})$ defined by

\begin{equation*}
\chi_{A_{\gamma}}(\vec{z}) = \begin{cases} \begin{array}{ll}
        1, & \text{ if  } \vec{z} \in A_{\gamma},\\
       0, & \text{ if  }  \vec{z} \in X \setminus A_\gamma,
        \end{array}
        \end{cases}
\end{equation*}
where $A_{\gamma}:=\left\lbrace \vec{z} : \gamma |\vec{z}| \geq g \right\rbrace $
\end{example}
\begin{proof}
\cite[Sec.3, Lemma 3.1]{Hintermuller}
\end{proof}

Finally, we introduce the mean value theorem for B-differentiable functions. 

\begin{thm}{(Mean value theorem for B-differentiable funcions.)}\label{thm:mvt} Let $D \subset \mathbb{R}^{n}$ be an open convex set, $J: D \rightarrow \mathbb{R}^{m}$ be a B-differentiable function, and $x_0,x_1 \in D$. The function $\varphi:[0,1] \rightarrow \mathbb{R}^{m}$ defined by $\varphi(t)=J'(x_0 +t(x_1 - x_0))(x_1-x_0)$ is Lebesgue integrable and
\[
J(x_1)=J(x_0)+ \displaystyle \int_0^1 J'(x_0 + t(x_1 - x_0))(x_1 - x_0) dt
\]
\end{thm}
\begin{proof}
See \cite[Prop. 3.1.1]{Scholtes}.
\end{proof}

\section{Regularization and Discretization}\label{sec:reganddisc}

The minimization problem  \eqref{eq:prob} involves a convex non-smooth functional. In fact, the norm $|\nabla y|$ in the second term implies that the functional in \eqref{eq:prob} is not differentiable. In this context, we propose a regularization approach, based on a local Huber regularization procedure. Huber regularization has been used in previous contributions to approximate numerically several free boundary and nonsmooth problems with similar structure (see \cite{Gonzalez1} and the references therein). 

 Let $\gamma > 0$. We introduce the function $\psi_{\gamma}: \mathbb{R}^n \rightarrow \mathbb{R}$ as follows:
 \begin{align*}
 \psi_{\gamma}: z \rightarrow \psi_{\gamma}(z)=
 \begin{cases}
  g |z | - \frac{g^2}{2 \gamma}  & if   |z|  > \frac{g}{\gamma}  \\
 \frac{\gamma}{2} |z |^2   & if   |z| \leq \frac{g}{\gamma}.
 \end{cases}
\end{align*}
The function  $\psi_{\gamma}$ corresponds to a local regularization of  the Euclidean norm. Thanks to this procedure we obtain the following regularized optimization problem 
\begin{equation}\label{eq:probreg}
\underset{u\in W_0^{1,p}(\Omega)} {\min} J_{\gamma}(u):=\frac{1}{p}\int_{\Omega}|\nabla u|^{p}\,dx + \int_{\Omega}\psi_{\gamma}(\nabla u)\,dx - \int_{\Omega} f u\,dx.
\end{equation}

\begin{thm}
Let $1<p< \infty$  and $\gamma > 0$. Then, problem \eqref{eq:probreg} has a unique solution $u_\gamma \in W_0^{1,p}(\Omega)$. Also, the sequence $\{u_\gamma\} \subset W_0^{1,p}(\Omega)$ converges strongly in $ W_0^{1,p}(\Omega)$ to the solution of problem \eqref{eq:prob}, as $\gamma \rightarrow \infty$.
\end{thm}
\begin{proof}
See \cite[Sec. 2]{Gonzalez1}.
\end{proof}

\subsection{Finite element approximation}\label{sec:FEM}

Let us introduce the finite element approximation of problem \eqref{eq:probreg}. Let $\Omega_h$ be  a given triangulation of the  domain $\Omega$ , $n_e \in \mathbb{N}$ be the number of triangles $T_i$  such that $\bar{\Omega}_h=\displaystyle \cup_{i=1}^{n_e} T_i$ and $N$ be the number of nodes of the triangulation $\Omega_h$. For any two triangles, their closures are either disjoint or have a common vertex or a common edge. Finally, let $\{P_j\}_{j=1,\cdots, N}$ be the vertices (nodes) associated with  $\Omega_h$. Taking this into account, we define 
\begin{equation*}
V_h:= \{v_h \in C(\bar{\Omega}_h): v_h|_{T_i} \in \mathbb{P}_1, \hspace{0.2cm} \forall T_i \in \Omega_h \},  
\end{equation*}
where $\mathbb{P}_1$ is the space of continuous piecewise linear functions defined on $\Omega_h$. Then the following space
\begin{equation}
V^0_h=W_0^{1,p}(\Omega) \cap V_h
\end{equation}
is the  finite-dimensional space associated with the triangulation $\Omega_h$. 

Considering the previous analysis, the finite element approximation of \eqref{eq:probreg} is formulated as follows
\begin{equation}\label{eq:probdis}
\underset{u_h\in V^0_h} {\min} J_{\gamma,h}(u_h):=\frac{1}{p}\int_{\Omega_h}|\nabla u_h|^{p}\,dx + \int_{\Omega_h}\psi_\gamma(\nabla u_h)\,dx - \int_{\Omega_h} f u_h\,dx. 
\end{equation}

\begin{remark}
The finite element approximation of problems like \eqref{eq:probreg} is restricted by the limited higher order regularity for the solution of the $p$-Laplacian (see \cite{Huang}). Due to this fact, this kind of problems are usually approximated by continuous piecewise linear elements, which we also implement in this paper. In \cite{BarretLiu,Bermejo,LiuBarret} the authors discuss optimal error estimates for sufficiently regular solutions, which can be obtained for specific data $f$ and $\Omega$. These results, however, are only valid for the $p$-Laplacian problem, i.e., when $g=0$. Since the solutions for variational inequalities of the second kind usually exhibit low global regularity, a deeper analysis is needed in order to obtain optimal error estimates for VIs. This, we consider, is beyond the scope of this paper.
\end{remark}

The convergence analysis of the multigrid algorithm is based on the differentiability properties of the functional $J_{\gamma,h}(u)$. In order to analyze the regularity of this discrete functional, we decompose it as follows
 \begin{equation}\label{eq:Jdiscrete_decompose}
J_{\gamma,h}(u_h):= \mathcal{F}_h(u_h) + \mathcal{G}_{\gamma,h}(\nabla u_h),
\end{equation}
where 
\begin{equation*}
\mathcal{F}_h(u_h):= \frac{1}{p}\int_{\Omega_h}|\nabla u_h|^{p}\,dx - \int_{\Omega_h} f u_h\,dx\,
\,\,\mbox{and}\,\,\, \mathcal{G}_{\gamma,h}(\nabla u_h):= \int_{\Omega_h}\psi_{\gamma}(\nabla u_h)\,dx.
\end{equation*}
$\mathcal{F}_h(u)$ is a functional associated with the discretized homogeneous Dirichlet  problem for the $p$-Laplace operator, and it is known to be a twice Fr\'echet-differentiable and strictly convex functional with Fr\'echet derivative $\mathcal{F}'_h(u_h)$ given by  
\begin{equation}\label{eq:F_prima}
\mathcal{F}'_h(u_h)v_h= \int_{\Omega_h} |\nabla u_h |^{p-2} \nabla u_h \cdot \nabla v_h \,dx - \int_{\Omega_h} f v_h\,dx,\,\,  \forall v_h \in  V^0_h
\end{equation}
and
\begin{equation}\label{eq:F_prima2}
\begin{array}{rll}
\mathcal{F}''_{h}(u_h)(v_h, w_h)=& \displaystyle \int_{\Omega_{h}} |\nabla u_h|^{p-2} \nabla v_h \cdot\nabla w_h \,dx\vspace{0.2cm}\\
& +(p-2) \int_{\Omega_{h}} |\nabla u_2|^{p-4} (\nabla u_h \cdot \nabla v_h) (\nabla u_h \cdot\nabla w_h)  , \,\,\, \forall v_h, w_h \in  V^0_h.
\end{array}
\end{equation}
 See (\cite{BarretLiu,Bermejo,Glowinski,Gonzalez1}) for further details.
\begin{prop}\label{prop}
Let $1<p< \infty$. The functional $J_{\gamma,h}(u_h)$ is 
differentiable with
\begin{equation}\label{eq:slant-derivative_0}
J'_{\gamma,h}(u_h)v_h:=\int_{\Omega_h}|\nabla u_h|^{p-2}\nabla u_h \cdot \nabla v_h\,dx + g \int_{\Omega_h} \frac{\gamma (\nabla u_h \cdot \nabla v_h)}{\max(g,\gamma |\nabla u_h|)}\,dx - \int_{\Omega_h} f v_h\,dx,\,\,  \forall v_h \in  V^0_h.
\end{equation}
Furthermore $J'_{\gamma,h}(u_h)$ is semismooth and its slant derivative is given by
\begin{equation}\label{eq:slant-derivative_1}
\begin{array}{lll}
(J'_{\gamma, h})^{\circ}(u_h)(v_h, w_h) := \displaystyle \int_{\Omega_{h}} |\nabla u_h|^{p-2} \nabla v_h \cdot \nabla w_h \,dx \vspace{0.2cm}\\\hspace{2cm}
+(p-2) \displaystyle \int_{\Omega_{h}} |\nabla u_h|^{p-4} (\nabla u_h \cdot \nabla v_h) (\nabla u_h \cdot\nabla w_h) \,dx\vspace{0.2cm}\\\hspace{3cm}
 + \displaystyle \int_{A_{\gamma}} g \frac{(\nabla v_h \cdot \nabla w_h)}{|\nabla u_h|}  \,dx - \displaystyle \int_{A_{\gamma}} g \frac{ (\nabla u_h \cdot \nabla v_h)(\nabla u_h \cdot \nabla w_h)}{|\nabla u_h|^3} \,dx \vspace{0.2cm}\\\hspace{5.5cm}
+ \displaystyle \int_{\Omega_{k-1} \setminus A_{\gamma}} \gamma (\nabla v_h \cdot \nabla w_h) \,dx, \,\,\,  \forall v_h, w_h \in  V^0_h.
\end{array}
\end{equation}
\end{prop}
\begin{proof}

Let us start by analyzing the functional $\mathcal{G}_{\gamma,h}(\nabla u_h)$.  It is known that $\mathcal{G}_{\gamma,h}$ is differentiable (see \cite[Sec. 2.2]{Gonzalez1}), and moreover, we know that
\begin{equation*}
\mathcal{G}'_{\gamma,h}(\nabla u_h)v_h= g \int_{A_{\gamma,h}} \frac{\nabla u_h \cdot \nabla v_h}{|\nabla u_h|}\,dx + g \int_{\Omega_h \setminus A_{\gamma,h}} \gamma (\nabla u_h \cdot \nabla v_h)\,dx, \,\,\,  \forall v_h \in  V^0_h,
\end{equation*}
where 
\[
A_{\gamma,h} =\{x \in \Omega_h \, : \, \gamma|\nabla u_h(x)| \geq g\}.
\]
By using the $\max$ function, we can rewrite $ \mathcal{G}'_{\gamma,h}(\nabla u_h)(v_h)$ in the following way.
\begin{equation}\label{Ghderiv}
\mathcal{G}'_{\gamma,h}(\nabla u_h)v_h:= g \int_{\Omega_h} \frac{\gamma (\nabla u_h \cdot \nabla v_h)}{\max(g,\gamma |\nabla u_h|)}\,dx, \,\,\,  \forall v_h \in  V^0_h.
\end{equation}
Next, from \eqref{eq:Jdiscrete_decompose}, it follows that
\begin{equation}\label{eq:Jdiscrete_decompose_grad}
J'_{\gamma,h}(u_h)v_h= \mathcal{F}'_h(u_h)v_h + \mathcal{G}'_{\gamma,h}(\nabla u_h)v_h,
\end{equation}
which, thanks to \eqref{eq:F_prima} and \eqref{Ghderiv}, implies that
\begin{equation*}
J'_{\gamma,h}(u_h)v_h:=\int_{\Omega_h}|\nabla u_h|^{p-2}\nabla u_h \cdot \nabla v_h\,dx + g \int_{\Omega_h} \frac{\gamma (\nabla u_h \cdot \nabla v_h)}{\max(g,\gamma |\nabla u_h|)}\,dx - \int_{\Omega_h} f v_h\,dx \,\,\,  \forall v_h \in  V^0_h.
\end{equation*}

The second derivative of $J_{\gamma,h}(u_h)$ does not exist. In fact, the functional $ \mathcal{G}'_{\gamma,h}(\nabla u_h)$ is not differentiable since this functional involves the $\max$ function. Fortunately, the $\max$ function is slantly differentiable when defined in finite dimensional spaces (see Example \ref{ex:max}). Thus, we can calculate the slant derivative of $\mathcal{G}'_{\gamma,h}(\nabla u_h)$, denoted  by $(\mathcal{G}'_{\gamma,h})^{\circ}(\nabla u_h)$, as follows.
\vspace{3mm}

\noindent $|\nabla u| \geq \frac{g}{\gamma}$: Here, we have that
\begin{equation*}
\small
\begin{array}{rll}
(\mathcal{G}'_{\gamma,h})^{\circ}(\nabla u_h)(v_h, w_h) &=& \displaystyle g \int_{A_{\gamma,h}} \frac{\gamma (\nabla v_h \cdot \nabla w_h)}{\max(g,\gamma |\nabla u_h|)}\,dx - g \int_{A_{\gamma,h}} \frac{\chi_{A_{\gamma,h}} (x) \cdot \gamma (\nabla u_h \cdot \nabla w_h)}{(\max(g,\gamma |\nabla u_h|))^2 |\nabla u_h|} \gamma (\nabla u_h \cdot \nabla v_h)\,dx \vspace{0.3cm}\\ 
 &=&\displaystyle g \int_{A_{\gamma,h}} \frac{\gamma (\nabla v_h \cdot \nabla w_h)}{\gamma |\nabla u_h|}\,dx   - g\int_{A_{\gamma,h}} \frac{ \gamma ^2 (\nabla u_h \cdot \nabla w_h) (\nabla u_h \cdot \nabla v_h)}{(\gamma |\nabla u_h|))^2 |\nabla u_h|} \,dx \vspace{0.3cm}\\
&=&\displaystyle g\int_{A_{\gamma,h}} \frac{(\nabla v_h \cdot \nabla w_h)}{ | \nabla u_h|}\,dx   - g \int_{A_{\gamma,h}} \frac{  (\nabla u_h \cdot \nabla w_h) (\nabla u_h \cdot \nabla v_h)}{ |\nabla u_h|^3 } \,dx,
\end{array}
\normalsize
\end{equation*}
where $\chi_{A_{\gamma,h}}$ is the slant derivative of the function $\max(g,\gamma |\nabla u_h|)$.\\\\
\noindent $|\nabla u| < \frac{g}{\gamma}$: Here, we have that
\begin{equation*}
\small
\begin{array}{lll}
(\mathcal{G}'_{\gamma,h})^{\circ}(\nabla u_h)(v_h, w_h) 
= \displaystyle g\int_{I_{\gamma,h}} \frac{\gamma (\nabla v_h \cdot \nabla w_h)}{\max(g,\gamma |\nabla u_h|)}\,dx \vspace{0.3cm} \\\hspace{4cm} -g \displaystyle \int_{I_{\gamma,h}} \frac{\chi_{A_{\gamma,h}} (x) \cdot \gamma (\nabla u_h \cdot \nabla w_h)}{(\max(g,\gamma |\nabla u_h|))^2 |\nabla u_h|} \gamma (\nabla u_h \cdot \nabla v_h)\,dx\vspace{0.3cm} \\\hspace{5cm} 
= \displaystyle g\int_{I_{\gamma,h}} \frac{\gamma (\nabla v_h \cdot \nabla w_h)}{g}\,dx  =\displaystyle \int_{I_{\gamma,h}} \gamma (\nabla v_h \cdot \nabla w_h)\,dx,
\end{array}
\normalsize
\end{equation*}
where $I_{\gamma,h}:=\Omega_h \setminus A_{\gamma,h}$.

Then, the slant derivative of $\mathcal{G}'_{\gamma,h}(\nabla u_h)$ reads as follows
\begin{equation}\label{eq:G_prima_circ}
\begin{array}{rll}
(\mathcal{G}'_{\gamma,h})^{\circ}(u_h)(v_h, w_h)  = \displaystyle g\int_{A_{\gamma,h}}  \frac{(\nabla v_h \cdot \nabla w_h)}{|\nabla u_h|} - g\int_{A_{\gamma,h}} \frac{(\nabla u_h \cdot \nabla v_h)(\nabla u_h \cdot \nabla w_h)}{|\nabla u_h|^3}\vspace{0.3cm}\\+ \displaystyle \int_{I_{\gamma,h}} \gamma (\nabla v_h \cdot \nabla w_h), \,\,\, \forall v_h, w_h \in  V^0_h.
\end{array}
\end{equation}
On the other hand, from \eqref{eq:F_prima2} we have that
\begin{equation}\label{eq:F_prima_prima}
\begin{array}{rll}
\mathcal{F}''_{h}(u_h)(v_h, w_h)=& \displaystyle \int_{\Omega_{h}} |\nabla u_h|^{p-2} \nabla v_h \cdot\nabla w_h \,dx\vspace{0.2cm}\\
& +(p-2) \int_{\Omega_{h}} |\nabla u_2|^{p-4} (\nabla u_h \cdot \nabla v_h) (\nabla u_h \cdot\nabla w_h)  , \,\,\, \forall v_h, w_h \in  V^0_h.
\end{array}
\end{equation}
Hence, from \eqref{eq:Jdiscrete_decompose_grad} and Proposition \ref{remark1}, we can state that
\begin{equation}\label{sec_derivative_decompose}
(J'_{\gamma,h})^{\circ}(u_h)(v_h, w_h)= \mathcal{F}''_{h}(u_h)(v_h , w_h) + (\mathcal{G}'_{\gamma,h})^{\circ}(\nabla u_h)(v_h ,  w_h)  \,\,\, \forall v_h, w_h \in  V^0_h,
\end{equation}
which, thanks to \eqref{eq:G_prima_circ}, \eqref{eq:F_prima_prima} and \eqref{sec_derivative_decompose}, yields that
\begin{equation}
\begin{array}{lll}
(J'_{\gamma,h})^{\circ}(u_h)(v_h, w_h)=\displaystyle \int_{\Omega_{h}} |\nabla u_h|^{p-2} \nabla v_h \nabla w_h \,dx \vspace{0.2cm}\\\hspace{1.5cm}+(p-2) \displaystyle \int_{\Omega_{h}} |\nabla u_h|^{p-4} (\nabla u_h \cdot \nabla v_h) (\nabla u_h \cdot\nabla w_h) \,dx\vspace{0.2cm}\\\hspace{2cm}
+g\displaystyle \int_{A_{\gamma,h}} \frac{(\nabla v_h \cdot \nabla w_h)}{|\nabla u_h|}  \,dx - g\int_{A_{\gamma,h}} \frac{\nabla u_h \cdot \nabla v_h(\nabla u_h \cdot \nabla w_h)}{|\nabla u_h|^3} \,dx \vspace{0.2cm}\\\hspace{3cm}
+ \gamma\displaystyle \int_{I_{\gamma,h}} (\nabla v_h \cdot \nabla w_h) \,dx, \,\,\,  \forall v_h, w_h \in  V^0_h.
\end{array}
\end{equation}
\end{proof}

\begin{prop}\label{prop2.0}
The slanting function $(J'_{\gamma,h})^{\circ}(u_h)$ is positive definite.
\end{prop}
\begin{proof}
Following the decompostition presented in \eqref{sec_derivative_decompose}, we know that 
\begin{equation}\label{eq:N-derivative}
(J'_{\gamma,h})^{\circ}(u_h)(w_{h}, w_{h})= \mathcal{F}''_{h}(u_{h})(w_{h}, w_{h})+ (\mathcal{G}'_{\gamma,h})^{\circ}(u_{h})(w_{h}, w_{h}), 
\end{equation}
It is well known that $\mathcal{F}$ is a strictly convex functional, which implies that 
\begin{equation}\label{eq:derivative_p_lapla}
\mathcal{F}''_{h}(u_{h})(w_{h}, w_{h}) > 0, \hspace{0.3cm} \forall w_{h} \in V^0_{h} \setminus{\{0\}}.
\end{equation}
Next, let us recall the expression $(\mathcal{G}'_{\gamma,h})^{\circ}(u_{h})(w_{h}, w_{h})$ given by
\begin{equation}\label{eq:derivative_gamma}
\begin{array}{rll}
(\mathcal{G}'_{\gamma,h})^{\circ}(u_{h})(w_{h}, w_{h}) &=& \displaystyle \int_{A_{\gamma,h}} g \frac{(\nabla w_{h} \cdot \nabla w_{h})}{|\nabla u_{h}|} -  \int_{A_{\gamma,h}} g \frac{(\nabla u_{h} \cdot \nabla w_{h})^2}{|\nabla u_{h}|^3}\vspace{0.2cm}\\
&&+ \displaystyle \int_{\Omega_{k} \setminus A_{\gamma,h}} \gamma (\nabla w_{h} \cdot \nabla w_{h}),  \,\, \forall w_{h} \in V^0_{h}.
\end{array}
\end{equation}
Applying Cauchy-Schwarz to the right hand side in \eqref{eq:derivative_gamma},  we have that
\begin{equation*}
\begin{array}{rll}
\displaystyle \int_{A_{\gamma,h}} g \frac{(\nabla u_{h} \cdot \nabla w_{h})^2}{|\nabla u_{h}|^3} & \leq& \displaystyle  \int_{A_{\gamma,h}} g \frac{|\nabla u_{h}|^2 |\nabla w_{h}| ^2}{|\nabla u_{h}|^3}\vspace{0.2cm}\\
&=&\displaystyle  \int_{A_{\gamma,h}} g  \frac{|\nabla w_{h}|^2}{|\nabla u_{h}|}\vspace{0.2cm}\\
&=&\displaystyle  \int_{A_{\gamma,h}} g  \frac{(\nabla w_{h} \cdot \nabla w_{h})}{|\nabla u_{h}|},
\end{array}
\end{equation*}
which implies that
\begin{equation}\label{eq:derivative_gamma_positive}
(\mathcal{G}'_{\gamma,h})^{\circ}(u_{h})(w_{h}, w_{h}) \geq  \int_{\Omega_{h} \setminus A_{\gamma,h}} \gamma (\nabla w_{h} \cdot \nabla w_{h}) > 0, \text{  since  } w_{h} \neq 0.
\end{equation}
Then, $(J'_{\gamma,h})^{\circ}(u_{h})(w_{h}, w_{h}) > 0$.
\end{proof}

\section{The Multigrid for Optimization (MG/OPT) Method}\label{sec:MGOPT}
In this section, we present the multigrid for optimization (MG/OPT) algorithm for solving the regularized and discretized optimization problem \eqref{eq:probdis}. The MG/OPT method was introduced  as an efficient  tool for large scale optimization problems (see \cite{Nash,Lewis_Nash}). The idea of the algorithm is to take advantage of the solutions of problems discretized in coarse meshes to optimize problems in finer meshes. The efficient resolution of coarse problems provides a way to calculate search directions for problems discretized in finer levels. 
 
In order to present the algorithm, we shall introduce the following preliminaries. Let $\Omega$ be a given bounded domain. A standard procedure to generate a sequence of triangulations on $\Omega$ is to define a coarse mesh and then refine it several times until getting the sequence $\{\Omega_k\}_{k=0,\ldots,m}$,  i.e., each $\Omega_k$ is obtained from $\Omega_{k-1}$ by regular subdivision. This procedure  joins the edge midpoints of any triangle in mesh $\Omega_{k-1}$ by edges, and forms the new triangles of  $\Omega_k$ (see Figure \ref{fig_mallas}). Each node in $\Omega_{k-1}$ is also a node in $\Omega_k$, and every node belonging to $\Omega_{k}$ but not to $\Omega_{k-1}$ is the midpoint of an edge in $\Omega_{k-1}$. Let us denote by $N^{(k)}$ the number of nodes associated to each $\Omega_k$. Thus, the nodes of $\Omega_{k-1}$ are the first $N^{(k-1)}$ nodes of $\Omega_k$ (see Figure \ref{fig_interp}).
 \begin{figure}[H]
\begin{minipage}[htb]{.3\textwidth}
\begin{center}
\includegraphics[scale=0.3]{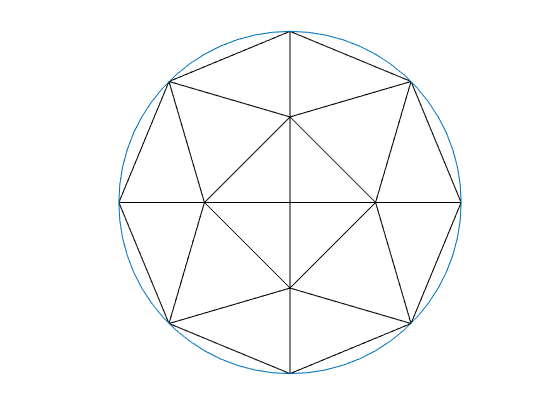}
\label{fig-ex1}
\end{center}
\end{minipage}
\begin{minipage}[htb]{.3\textwidth}
\begin{center}
\includegraphics[scale=0.3]{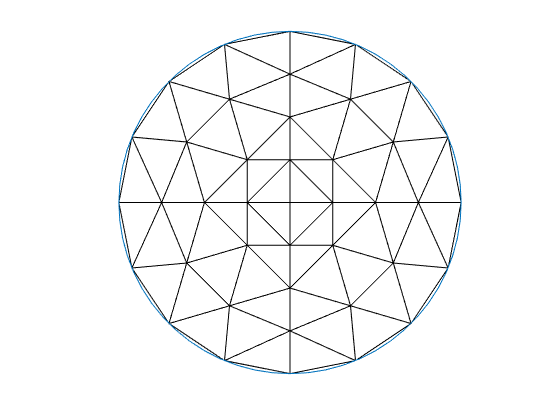}
\end{center}
\end{minipage}
\begin{minipage}[htb]{.3\textwidth}
\begin{center}
\includegraphics[scale=0.3]{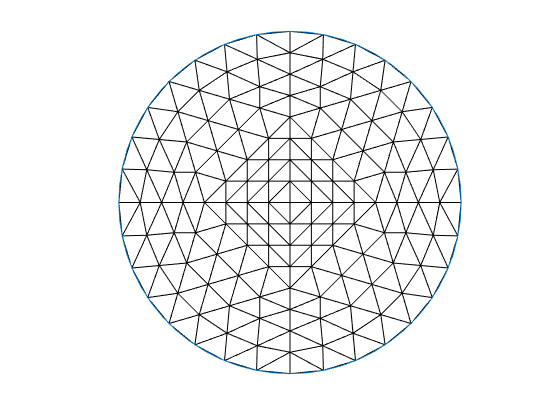}
\centering
\label{fig-ex3}
\end{center}
\end{minipage}
\centering
\caption{Regular subdivision. From left to right: $\Omega_1, \Omega_2$ and $\Omega_3$}\label{fig_mallas}
\end{figure}
When working with domains with curved boundaries, the refinement procedure described above is slightly different in the boundary edges. This is the case, for instance, of the domain depicted in  Figure \ref{fig_interp}. During the refinement, instead of generating a new midpoint in the nearest triangulation edges to the boundary, each edge is replaced by two edges that intersect at the midpoint of the curved-segment of the boundary, i.e., the boundary is bisected by the two new edges added. Consequently, the resulting triangulation only covers the domain $\Omega$ approximately, introducing a source of error \cite[Ch. 4, pp. 93]{Gockenbach}. Since a non polygonal domain can not be triangulated exactly, approximating a domain with a curved boundary is a matter under research. One interesting technique to analyze this problem is the isoparametric method which uses finite elements with curved edges \cite[Ch.4, Sec. 4.7]{Gockenbach}. This approach seems to be the most appropriate for problems of fluid mechanics, which is why it will be considered in future contributions.

The multigrid approach involves two transfer operators. As we are working with a set of meshes and  the algorithm runs at each level of discretization, we need to transfer information among  the different grids. Then, we intoduce the \textit{fine-to-coarse grid transfer operator}, $I_k^{k-1}$, and the \textit{coarse-to-fine grid transfer operator}, $I_{k-1}^k$. 

The latter operator transfers information from a coarse mesh  $\Omega_{k-1}$ to a finer mesh  $\Omega_{k}$. It is also called the prolongation operator. On the other hand, $I_k^{k-1}$ or restriction operator, transfers information from a fine grid to a coarse one. 

In this paper we use the mesh data structure and the operators implemented by M. S. Gockenbach  in order to obtain $I_k^{k-1}$ and $I_{k-1}^{k}$ (see \cite[Ch. 6, Ch. 13]{Gockenbach}). In what follows, the ideas stated in the previous reference are outlined for the prolongation and restriction procedures in a coarse grid $\Omega_{k-1}$ and a finer grid $\Omega_{k}$. For instance, in Figure \ref{fig_interp} we have  the first stage of refinement in a disk domain  with  grids $\Omega_0$ and $\Omega_1$ .
\begin{figure}[H]
\begin{minipage}[l]{0.4\textwidth}
\begin{center}
\includegraphics[trim = 130mm 7mm 40mm 2mm, clip, scale=0.4]{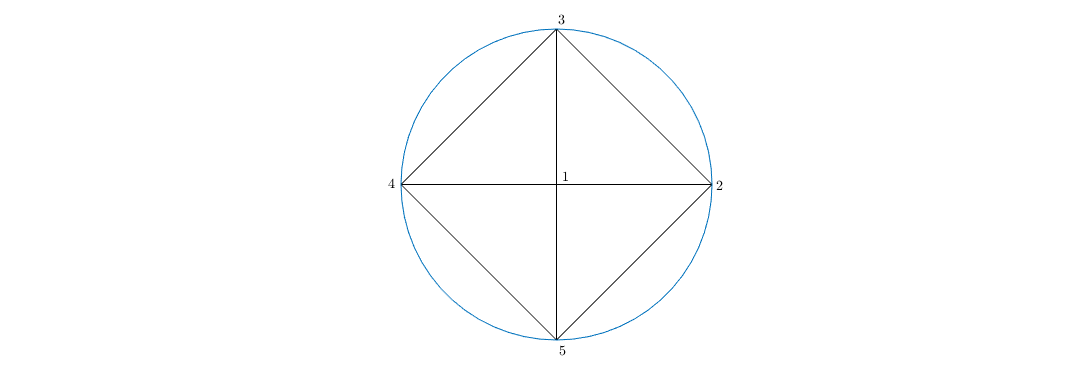}
\end{center}
\end{minipage}
\begin{minipage}[r]{0.4\textwidth}
\begin{center}
\includegraphics[trim = 130mm 5mm 40mm 5mm, clip, scale=0.4]{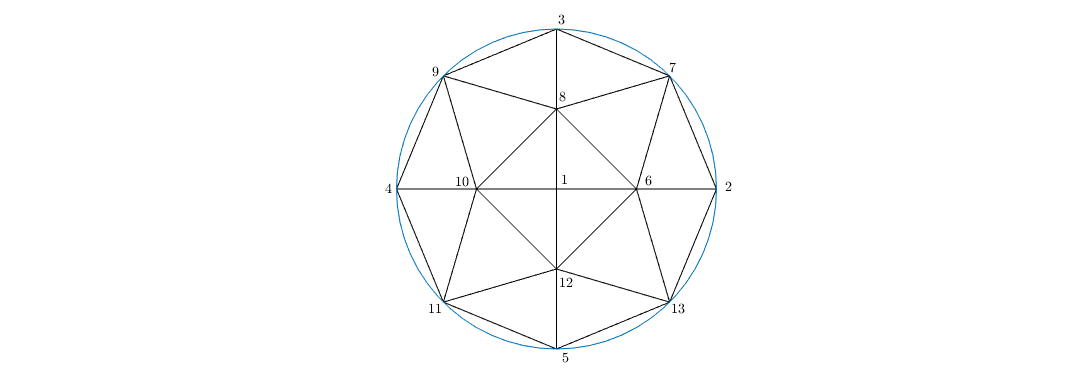}
\end{center}
\end{minipage}
\centering
\caption{ $\Omega_0$ (left), $\Omega_1$ (right).}\label{fig_interp}
\end{figure}
In order to transfer information from the nodes in a coarse mesh $\Omega_{k-1}$ (e.g. $\Omega_0$ in Figure \ref{fig_interp}) to a refined mesh $\Omega_k$ (e.g. $\Omega_1$), we first copy the nodal values of the nodes in $\Omega_{k-1}$ to the corresponding nodes in $\Omega_{k}$ (in our example they correspond to the nodes with indexes $1,2,3,4$ and $5$). By construction, we know that these are the first $N^{(k-1)}$ ($N^{(0)}=5$) nodes of $\Omega_k$, and,the midpoints complete the $N^{(k)}$ ($N^{(1)}=13$) nodes of the mesh $\Omega_k$ (in $\Omega_1$ they correspond to the  nodes with indexes  $6, 7, \cdots, 13$). Then, since the values for nodes with index $j=1, \cdots N^{(k-1)}$ are already given, we need to compute the nodal values for the midpoints, i.e., for nodes with index $j= N^{(k-1)}+1, \cdots, N^{(k)}$.
Let $V(\textbf{P})_k=\{V(P_j)\}_{j=1,\cdots, N^{(k)}}$ be the vector containing the nodal values associated to  $\Omega_k$ 
The prolongation operator computes the values to the midpoints  as follows:
\[
V(P_j)=\frac{1}{2}\left(V(P_{end(j,1)})+V(P_{end(j,2)})\right), \hspace{2mm} j>N^{(k-1)},
\]
where the nodes $P_{end(j,1)}$ and $P_{end(j,2)}$ are the endpoints of the edge in $\Omega_{k-1}$ for which $P_j$ is the midpoint. These endpoints are called node parents and, by construction, are unique for each node.

Now, to transfer information from the nodes in $\Omega_k$ to $\Omega_{k-1}$ (e.g. $\Omega_1$ to $\Omega_{0}$ in Figure \ref{fig_interp}) the restriction operator first copy the values of the nodes in $\Omega_k$ to the corresponding nodes that are also in $\Omega_{k-1}$. Then, for $j>N^{k-1}$, the operator perform the following computation
\[
V(P_{end(j,1)})=V(P_{end(j,1)})+\frac{1}{2}V(P_j) \hspace{3mm} \text{and} \hspace{3mm} V(P_{end(j,2)})=V(P_{end(j,2)})+\frac{1}{2}	V(P_j).
\]

The mesh data structure implemented allows us to copy and extract  nodes from any specific discretization. We refer the reader to see \cite[Ch. 13, sec. 13.2.1]{Gockenbach} for a detailed explanation of the transfer operators.

In multigrid schemes is standard to assume that
\begin{equation*}
I_k^{k-1} = c \left( I_{k-1}^k \right)^{T},
\end{equation*}
where $c$ is a constant. In our case, if we compute the products $I_{k-1}^k V(\textbf{P})_{k-1}=V(\textbf{P})_{k}$ and $I_{k}^{k-1} V(\textbf{P})_{k}=V(\textbf{P})_{k-1}$, the restriction and prolongation operators satisfy the condition (see \cite[Sec. 13.2.1, p. 294]{Gockenbach})
\begin{equation}\label{eq:condition}
I_k^{k-1} = \left( I_{k-1}^k \right)^{T}.
\end{equation}

Now that we have introduced the interpolation operators, we are ready to discuss the MG/OPT algorithm for problem \eqref{eq:probdis}. The MG/OPT method corresponds to a nonlinear programming adaptation of the \textit{full approximation storage} (FAS) scheme (see \cite{Brandt,Trottenberg}). The multigrid subproblems arising from the different discretization levels are nonlinear optimization problems \cite{Lewis_Nash}. Then,  MG/OPT is related to different optimization techniques ranging from the gradient method to quasi Newton methods to solve the problems at each level. The multigrid for optimization approach requires mild conditions regarding the underlying optimization algorithm. Mainly, this algorithm needs to be globally convergent. However, it is important to highlight that at each level of discretization we need to find a sufficiently accurate solution for the minimization subproblem. Then, the selection of the underlying optimization algorithm is not trivial and depends on the inner characteristics of the optimization problem. 

In the present problem, since the $p$-Laplacian is involved in the objective functional $J_{\gamma}$, we have to consider that its finite element approximation \eqref{eq:probdis} results in a nonlinear and possibly degenerate finite dimensional problem (this may be the case when $p<2$, see \cite{Huang} and the references therein). Also, the functional $J_{\gamma,h}$ involves a semismooth function. These facts need to be taken into account when proposing the underlying optimization  algorithms. In our case, we propose a class of preconditioned descent algorithms, designed specifically for $p$-Laplacian type problems. For the convenience of the reader, in the next section we describe briefly this algorithm as is implemented in our computational results.

As we mentioned before, the main idea of the  MG/OPT algorithm is to use coarse  problems to generate, recursively, search directions for finer  problems. Then, a line search procedure, along with the underlying optimization algorithm is used to improve the solution at each level of discretization.

In what follows we present the MG/OPT algorithm. The underlying optimization algorithm will be denoted by $S_{opt}$ inside the multigrid approach. The initial discretized problem is given on the finest grid.  To facilitate the  implementation of the algorithm, the MG/OPT scheme is presented  in a recursive formulation.  Hence, we introduce the following slightly different notation for the optimization problem
\begin{equation*}
 \underset{u_k} {\min} \left( \hat{J}_{\gamma,k}(u_k) - \hat{f}_k^{\top}u_k\right).
\end{equation*}
We set $\hat{f}_k=0$ at the finest  level $k=m$. Then, $\hat{J}_{\gamma,k}$ corresponds to the functional $J_{\gamma,h}$ introduced in problem \eqref{eq:probdis}, discretized at each level $k=0,\ldots,m$. Therefore, we replace the subscript $h$ by $k$. Hereafter, the same treatment is given to $u_k$, which stands for $u_{h}$, discretized at each level $k$. 

\begin{remark}\label{notation}
Since $J_{\gamma,k}$ is Fr\'echet differentiable and defined in finite dimensional spaces, we can associate the derivative $J'_{\gamma,k}(u_k)$ with the gradient $\nabla J_{\gamma,k}(u_h)$ as follows (see \cite[Sec.1.1, p. 8]{Giaquinta})
\[
J'_{\gamma,k}(u_k)v_k=\nabla J_{\gamma,k}(u_k)^{\top}v_k,\,\, \forall v_k \in V_k^0.
\]
In what follows, we will use this representation.
\end{remark}

Summarizing and taking into account Remark \ref{notation}, the algorithm reads as follows.
\begin{algorithm}[H]
\caption{MG/OPT recursive$(\nu_1,\nu_2)$.}\label{alg:MGOPT_alg}
\begin{algorithmic}
\If {$k=0$}, solve $\min_{u_k} \left( \hat{J}_{\gamma,k}(u_k) - \hat{f}_k^{\top}u_k\right)$ and {return}.
\EndIf
\State Otherwise, $k > 0$. \\
Pre-optimization: Apply $\nu_1$ iterations of the optimization algorithm to the problem at level $k$.
\[u_k^{\ell}=S_{opt}(u_k^{\ell-1}),\,\,\,\ell=1,\ldots,\nu_1.\]
\State Coarse-grid correction.\\
\begin{itemize}
\item Restrict: $u_{k-1}^{\nu_1}= I_k^{k-1} u_k^{\nu_1}$.
\item Compute the fine-to-coarse gradient correction:
\[
\tau_{k-1}:= \nabla \hat{J}_{\gamma,k-1}(u_{k-1}^{\nu_1}) - I_k^{k-1} \nabla \hat{J}_{\gamma,k}(u_k^{\nu_1}).
\]
\item Define $\hat{f}_{k-1}:=I_k^{k-1} \hat{f}_k + \tau_{k-1}$

and apply one cycle of MGOPT$(\nu_1,\nu_2)$ to
\[
\min_{u_{k-1}} \left( \hat{J}_{\gamma,k-1} (u_{k-1}) - \hat{f}_{k-1}^{\top}u_{k-1} \right)
\]
to obtain $\widetilde{u}_{k-1}$.
\end{itemize}
\State Coarse-to-fine minimization.
\begin{itemize}
\item Prolongate error: $e:= I_{k-1}^k (\widetilde{u}_{k-1} - u_{k-1}^{\nu_1})$.
\item Line search in $e$ direction to obtain a step size $\alpha_k$.
\item Calculate the coarse-to-fine minimization step: $u_k^{\nu_1 +1}= u_k^{\nu_1} + \alpha_k e$.
\end{itemize}
\item Post-optimization:  Apply $\nu_2$ iterations of the optimization algorithm to the problem at level $k$.
\[u_k^{\ell}=S_{opt}(u_k^{\ell-1}),\,\,\,\ell=\nu_1+2,\ldots,\nu_1+\nu_2 +1.\]
\end{algorithmic}
\end{algorithm}
 The algorithm presented above contemplates one iteration of a V-cycle initialized with a  rough estimate of the solution on the finest grid. 
 
\section{Convergence Analysis}\label{sec:convergence}
In this section, we discuss the convergence properties of Algorithm \ref{alg:MGOPT_alg}. Following \cite{Lewis_Nash,Nash}, we can state that the global convergence of the underlying optimization algorithm ensures global convergence of the MG/OPT method. This comes from the fact that if we have an approximate solution (given by the underlying optimization algorithm) at each discretization level, the algorithm generates search directions for problems discretized on finer meshes. Once we have the descent direction, a line search procedure is used to improve the solution at each finer problem. 

In the classical convergence analysis of the MG/OPT methods, the three following conditions are critical (see \cite{Lewis_Nash,Nash,BorziVallejo}).

\begin{enumerate}
\item The discretized objective functional is convex at each level of discretization.
\item The subproblems
\begin{equation}\label{eq:subproblems}
\min_{u_{k-1}} \left( J_{\gamma,k-1} (u_{k-1}) - \tau_{k-1} ^{\top }u_{k-1} \right)
\end{equation}
are solved accurately enough.
\item The transfer operators satisfy  the standard condition $I_k^{k-1} = c \left( I_{k-1}^k \right)^{T}$.
\end{enumerate}
These  conditions are helpful to prove that the search direction provided by the MG/OPT algorithm is indeed a descent direction. For instance, the convexity condition is key to prove that the Hessian is positive definite. However, in our case, the classical Hessian does not exist. Thus, we  use the slant differentiability of the functional $J_{\gamma,h}$ to obtain positive definiteness, see Proposition \ref{prop2.0}.

Next, we comment on the assumptions on our problem. We are considering the subproblems \eqref{eq:subproblems} instead of the subproblems $\displaystyle\min_{u_{k-1}} \left( \hat{J}_{\gamma,k-1} (u_{k-1}) - \hat{f}^{\top}u_{k-1} \right)$ presented in Algorithm \ref{alg:MGOPT_alg}. The artificial term $I_k^{k-1} \hat{f}_k$ is simply introduced in order to facilitate the recursive implementation of the method. In fact, at the very beginning of the algorithm, on the finest mesh, this term is set to zero. In the succeeding iterations, $I_k^{k-1} \hat{f}_k$ corresponds to the recursive sum of the previous restricted $\tau_{k-1}$. Thus, in order to make the subsequent analysis easier, we analyze the subproblems  \eqref{eq:subproblems}. Also, since problem \eqref{eq:probdis} is strictly convex in $V_h^0$, without loss of generality, subproblems \eqref{eq:subproblems} are strictly convex in $V_{k}^0$, for $k=0,\ldots,m$.

Since we perform a few iterations of a suitable globally convergent optimization algorithm ($S_{opt}$), we ensure that the subproblems $\min_{u_{k-1}} \left( J_{\gamma,k-1} (u_{k-1}) - \tau _{k-1}^{\top}u_{k-1} \right)$ are solved accurately enough. Finally, \eqref{eq:condition} yields that $I_k^{k-1} =  \left( I_{k-1}^k \right)^{T}$.  

Let us recall that the search direction for the  MG/OPT algorithm is denoted by $e$ and search directions of the underlying optimization algorithm, (inside the MG/OPT loop) are denoted by $w_k$. 
 
\begin{prop}\label{e_desc_dir}
The search direction $e= I_{k-1}^k (\widetilde{u}_{k-1}-u_{k-1}^{\nu_1})$ is a descent direction for all $k=1, \ldots, m.$
\end{prop} 
\begin{proof}
We have to prove that
\begin{equation}\label{eq:descent}
\nabla J_{\gamma,k} (u_k^{\nu_1})^{\top} e < 0 , \hspace{0.3cm} \forall\, k=1,\ldots, m.
\end{equation}
From this point, for the readability of the proof, we drop the subscript  $\gamma$. First note that, if we solve
\[
\min_{u_{k-1}} \left( J_{k-1} (u_{k-1}) - \tau_{k-1}^{\top}u_{k-1} \right)
\]
exactly, then
\begin{equation*}
 \nabla J_{k-1} (\widetilde{u}_{k-1}) - \tau_{k-1}=0.
\end{equation*}
Since we are solving the problem inexactly (but accurately enough), we then have that,
\begin{equation}\label{eq:a}
 \nabla J_{k-1} (\widetilde{u}_{k-1}) - \tau_{k-1}=z,
\end{equation}
for some $z$ as small as the algorithm accuracy allows for. From Algorithm \ref{alg:MGOPT_alg} we have that 
\[
\tau_{k-1}:= \nabla J_{k-1}(u_{k-1}^{\nu_1}) - I_k^{k-1} \nabla J_k(u_k^{\nu_1}).
\]
Hence, we can rewrite \eqref{eq:a} as follows
\begin{equation}\label{eq:tau}
 \nabla J_{k-1} (\widetilde{u}_{k-1})= \nabla J_{k-1}(u^{\nu_1}_{k-1}) - I_{k}^{k-1}  \nabla J_{k}(u_{k}^{\nu_1}) +z.
\end{equation}
This expression implies that
\begin{equation}\label{eq:descent2}
\begin{array}{rll}
\nabla J_{k}(u_k^{\nu_1})^{\top} e&=&  \nabla J_{k}(u_k^{\nu_1})^{\top} I_{k-1}^k (\widetilde{u}_{k-1}-u_{k-1}^{\nu_1}) \vspace{0.2cm} \\ 
&=&  \nabla J_{k} (u_k^{\nu_1})^{\top} (I_{k}^{k-1})^{\top}(\widetilde{u}_{k-1}-u_{k-1}^{\nu_1}) \vspace{0.2cm} \\ 
&=&(I_k^{k-1} \nabla J_{k}(u_k^{\nu_1}))^{\top}(\widetilde{u}_{k-1}-u_{k-1}^{\nu_1}) \vspace{0.2cm} \\ 
&=& (\nabla J_{k-1}(u_{k-1}^{\nu_1})- \nabla J_{k-1}(\widetilde{u}_{k-1}) + z)^{\top}(w_{k-1}) \vspace{0.2cm} \\ 
&=& (\nabla J_{k-1}(u_{k-1}^{\nu_1})- \nabla J_{k-1}(\widetilde{u}_{k-1}))^{\top}(w_{k-1}) +z^{\top}w_{k-1}, 
\end{array}
\end{equation}
where
\[w_{k-1}=\widetilde{u}_{k-1}-u_{k-1}^{\nu_1}.\]
Next, let us focus on the  first term in the right-hand side of \eqref{eq:descent2}. Hereafter, we use the notation $\nabla J_{k-1}=H_{k-1}$. Then, we have that
\begin{equation}
( \nabla J_{k-1}(u_{k-1}^{\nu_1})- \nabla J_{k-1}(\widetilde{u}_{k-1}))^{\top}w_{k-1} =  H_{k-1}(u_{k-1}^{\nu_1})^{\top}w_{k-1} - H_{k-1}(\widetilde{u}_{k-1})^{\top} w_{k-1}.  
\end{equation}
We know, from Proposition \ref{prop}, that $ H_{k-1}$ is a semismooth function, which, thanks to Remark \ref{ssm_is_boul}, implies that $ H_{k-1}$ is B-differentiable. Thus, Theorem \ref{thm:mvt} yields that
\begin{small}
\begin{equation}\label{eq:left-hand_side}
\begin{array}{cl}
-(H_{k-1}(u_{k-1}^{\nu_1})- H_{k-1}(\widetilde{u}_{k-1}))^\top w_{k-1}=(H_{k-1}(\widetilde{u}_{k-1})- H_{k-1}(u_{k-1}^{\nu_1}))^\top w_{k-1} \vspace{0.2cm}\\\hspace{2.cm}
= \left( \displaystyle \int_0^1 H_{k-1}'(u_{k-1}^{\nu_1} + t (w_{k-1}))(w_{k-1}) dt \right)^{\top} w_{k-1}
=H_{k-1}^{\circ}(\widetilde{u}_{k-1})(w_{k-1},w_{k-1} )\vspace{0.2cm}\\\hspace{4.cm}
+\left( \displaystyle \int_0^1 H_{k-1}'(u_{k-1}^{\nu_1} + t (w_{k-1}))(w_{k-1})- H_{k-1}^{\circ}(\widetilde{u}_{k-1})(w_{k-1}) dt \right)^{\top} w_{k-1},
\end{array}
\end{equation}
\end{small}
where $H_{k-1}'(u_{k-1}^{\nu_1} + t (w_{k-1}))(w_{k-1})$ stands for the directional derivative of the operator $H_{k-1}$ at $(u_{k-1}^{\nu_1} + t (w_{k-1}))$ in the direction $w_{k-1}$. Furthermore,  $H_{k-1}^{\circ}(\widetilde{u}_{k-1})(w_{k-1},w_{k-1} )$  is given by \eqref{eq:slant-derivative_1}. Following \eqref{sec_derivative_decompose} and Proposition \ref{prop2.0}, we know that  $H_{k-1}^{\circ}(\widetilde{u}_{k-1})$ is definite positive. Consequently, there exists a constant $c$ such that
\begin{equation}\label{H_pd}
H_{k-1}^{\circ}(\widetilde{u}_{k-1})(w_{k-1},w_{k-1}) \geq c \|w_{k-1}\|^2.
\end{equation}
Next, let us focus on the last term on the right hand side of \eqref{eq:left-hand_side}. By using the  Cauchy-Schwarz inequality and the fact that $H_{k-1}'(u_{k-1}^{\nu_1} + t (w_{k-1}))(w_{k-1})$ is Lebesgue integrable, we obtain that
\begin{equation}\label{eq:int}
\begin{array}{lll}
\left( \displaystyle \int_0^1 H_{k-1}'(u_{k-1}^{\nu_1} + t (w_{k-1}))(w_{k-1})- H_{k-1}^{\circ}(\widetilde{u}_{k-1})(w_{k-1}) dt \right)^{\top} w_{k-1} \geq \vspace{0.2cm}\\\hspace{1cm}
- \displaystyle \Big| \left( \displaystyle \int_0^1 H_{k-1}'(u_{k-1}^{\nu_1} + t (w_{k-1}))(w_{k-1})- H_{k-1}^{\circ}(\widetilde{u}_{k-1})(w_{k-1}) dt \right)^{\top} w_{k-1} \Big| \geq \vspace{0.2cm}\\\hspace{1.5cm}
- \displaystyle \Big\| \left( \displaystyle \int_0^1 H_{k-1}'(u_{k-1}^{\nu_1} + t (w_{k-1}))(w_{k-1})- H_{k-1}^{\circ}(\widetilde{u}_{k-1})(w_{k-1}) dt \right)  \Big\|  \Big\|w_{k-1} \Big\| \geq \vspace{0.2cm}\\\hspace{2.5cm}
- \displaystyle  \left( \displaystyle \int_0^1  \Big\| H_{k-1}'(u_{k-1}^{\nu_1} + t (w_{k-1}))(w_{k-1})- H_{k-1}^{\circ}(\widetilde{u}_{k-1})(w_{k-1}) \Big\|  dt \right)   \Big\|w_{k-1} \Big\|.
\end{array}
\end{equation}
Moreover, since  $ H_{k-1}$ is semismooth, from Propositions \ref{ssm_prop} and \ref{b-prop2}  we have that
\begin{equation*}
\begin{array}{ll}
\Big\|H_{k-1}'(u_{k-1}^{\nu_1} + t (w_{k-1}))(w_{k-1})- H_{k-1}^{\circ}(\widetilde{u}_{k-1})(w_{k-1}) \Big\|\vspace{0.2cm}\\\hspace{2.cm}
\leq \Big\| H_{k-1}'(u_{k-1}^{\nu_1} + t (w_{k-1}))(w_{k-1})- H_{k-1}'(u_{k-1}^{\nu_1})(w_{k-1}) \Big\| \vspace{0.2cm}\\\hspace{3.5cm}
+  \Big\| H_{k-1}'(u_{k-1}^{\nu_1})(w_{k-1}) - H_{k-1}^{\circ} (\widetilde{u}_{k-1})(w_{k-1})    \Big\| 
=o(\| w_{k-1}\|).
\end{array}
\end{equation*}
Hence, for an arbitrary $\epsilon> 0$, it holds that
\begin{equation*}
-\Big\|H_{k-1}'(u_{k-1}^{\nu_1} + t (w_{k-1}))(w_{k-1})- H_{k-1}^{\circ}(\widetilde{u}_{k-1})(w_{k-1}) \Big\| \geq -\epsilon \|w_{k-1}\|.
\end{equation*}
Thus, from \eqref{eq:int} we have that
\begin{equation}\label{ssm_cond}
\left( \displaystyle \int_0^1 H_{k-1}'(u_{k-1}^{\nu_1} + t (w_{k-1}))(w_{k-1})-  H_{k-1}^{\circ}(\widetilde{u}_{k-1})(w_{k-1}) dt \right)^{\top} w_{k-1} \geq -\epsilon \|w_{k-1}\|^2. 
\end{equation}
Finally, by taking $\epsilon < \frac{c}{2}$ and considering \eqref{eq:left-hand_side}, \eqref{H_pd} and \eqref{ssm_cond}, we conclude that
\begin{equation}\label{cond_fin}
\begin{array}{cl}
-(H_{k-1}(u_{k-1}^{\nu_1})- H_{k-1}(\widetilde{u}_{k-1}))^\top w_{k-1} &\geq c\|w_{k-1}\|^2 - \epsilon \|w_{k-1}\|^2\\
&=(c - \epsilon)\|w_{k-1}\|^2\\
& > \displaystyle \frac{c}{2} \|w_{k-1}\|^2\\
& >0.
\end{array}
\end{equation}
Note that $w_{k-1}=\widetilde{u}_{k-1}-u^{\nu_1}_{k-1} \neq 0$. Consequently, we have that

\begin{equation}\label{eq:sdp}
(H_{k-1}(u_{k-1}^{\nu_1})- H_{k-1}(\widetilde{u}_{k-1}))^{\top}w_{k-1} < 0.
\end{equation}
In order to prove that $e$ is a descent direction, we still need to prove that the third term of the right hand side in \eqref{eq:descent2} satifies that
\begin{equation*}
 z ^{\top}w_{k-1}= z^{\top}(\widetilde{u}_{k-1}-u_{k-1}^{\nu_1}) < 0.
\end{equation*}
Note that $\widetilde{u}_{k-1}$ is the solution of the problem 
\begin{equation*}
\min_{u_{k-1}} \left( \hat{J}_{k-1} (u_{k-1}) - \tau^{\top}u_{k-1}\right).
\end{equation*}
Therefore, 
\begin{equation*}
J_{k-1} (\widetilde{u}_{k-1}) - \tau_{k-1}^{\top}\widetilde{u}_{k-1} < J_{k-1} (u_{k-1}^{\nu_1}) - \tau_{k-1}^{\top}u_{k-1}^{\nu_1},
\end{equation*}
which is equivalent to 
\begin{equation}\label{eq:zw_prima}
J_{k-1} (\widetilde{u}_{k-1}) - J_{k-1} (u_{k-1}^{\nu_1}) <  \tau_{k-1}^{\top}(\widetilde{u}_{k-1} -u_{k-1}^{\nu_1}),
\end{equation}
since  the optimization algorithm was initialized with $u_{k-1}^{\nu_1}$.  Now, using the mean value theorem for differentiable functionals we have that
\begin{equation}\label{eq:zw}
(J_{k-1} (\widetilde{u}_{k-1}) - J_{k-1} (u_{k-1}^{\nu_1}))= \nabla J_{k-1} (\xi)^{\top}(w_{k-1}),
\end{equation}
for some $\xi$ between $\widetilde{u}_{k-1}$ and $u_{k-1}^{\nu_1}$.  Hence, from the inequality  \eqref{eq:zw_prima} and equation \eqref{eq:zw} we have that
\begin{equation*}
\begin{array}{cl}
\nabla J_{k-1}(\xi)^{\top}(w_{k-1}) &< \tau_{k-1}^{\top}(\widetilde{u}_{k-1} - u_{k-1}^{\nu_1})\vspace{0.2cm}\\
 &=\tau_{k-1}^{\top}w_{k-1}
 \end{array}
 \end{equation*}
 which implies that
 \begin{equation}\label{eq:last}
 \nabla J_{k-1} (\xi)^{\top}w_{k-1} -\tau_{k-1}^{\top}w_{k-1} < 0.
 \end{equation}
Next, by approximating $\xi \approx \widetilde{u}_{k-1} $, from \eqref{eq:a} and \eqref{eq:last}, we obtain that

\begin{equation}\label{eq:last_last}
z^{\top}w_{k-1} \lesssim 0.
\end{equation}
Summarizing, \eqref{eq:descent2}, \eqref{eq:sdp} and \eqref{eq:last_last}  imply that
\begin{equation*}
\nabla J_{k}(u_k^{\nu_1})^{\top} e <0, \forall k=0, \cdots, m.
\end{equation*}
and we can conclude that $e$ is a descent direction.
\end{proof}

Finally, thanks to the previous results, we can state and prove the following theorem of convergence for the MG/OPT Algorithm \ref{alg:MGOPT_alg}.

\begin{thm}
Suppose that the following hypotheses are satisfied:
\begin{itemize} 
\item The optimization algorithm, $S_{opt}$, applied to an optimization problem of any resolution, is globally convergent, i.e.,
\begin{equation}\label{global_convergence}
\lim_{r \rightarrow \infty} \| \nabla J_{\gamma,h}(u_{h_r}) \| = 0.
\end{equation}

\item At least one of the parameters $\nu_1$ or $\nu_2$ is positive.
\item The search direction $e= I_{k-1}^k (\widetilde{u}_{k-1}-u_{k-1}^{\nu_1})$ is a descent direction.
\end{itemize}
Then the MG/OPT algorithm is globally convergent in the sense of \eqref{global_convergence}.
\end{thm}

\begin{proof}
Let us start by noticing that if $\nu_1$ or $\nu_2$ are positive, at least one iteration of the optimization algorithm is performed at every cycle of  MG/OPT. Thus, an approximate solution at each level of discretization $k$ is obtained. Since the search direction $e=I_{k-1}^k (\widetilde{u}_{k-1}-u_{k-1}^{\nu_1}) $ is a descent direction, the  approximate solution given at each level improves at every cycle of  MG/OPT, i.e., the functional value $\hat{J}_{\gamma,k}(u_k)$ decreases at each cycle after the solution update  $u_k^{\nu_1 +1}= u_k^{\nu_1} + \alpha_k e$. Consequently, as the underlying optimization algorithm is globally convergent, the multigrid optimization algorithm is globally convergent.
\end{proof}

\begin{remark}
The convergence of MG/OPT methods depends on the behavior of the underlying algorithms. In fact, usually the performance of the selected optimization method is the same that is verified for the MG/OPT algorithm (see \cite{Borzi,Lewis_Nash,Nash, BorziVallejo}). Moreover, it is common that for optimization methods designed for nonsmooth problems, only global convergence results can be established. In our particular selection for the underlying algorithm, the global convergence of the preconditioned descent method is guaranteed \cite{Gonzalez1}. However, the rates of this convergence are not investigated. To the best of our knowledge, the use of  semismooth derivatives, in the context of multigrid methods, is a novel perspective. In the present case, we need to analyse the behavior of the ``slant Hessian'' $(J'_{\gamma,h})^{\circ}$ in the context of the underlying algorithms.  This idea is currently under study and it will be developed in future contributions.
\end{remark}

\section{Implementation}\label{sec:impl}
\subsection{Optimization algorithm }
In this section, we briefly discuss the preconditioned descent algorithm proposed in \cite{Gonzalez1}, which is proposed for the underlying optimization algorithm within the MG/OPT algorithm.

Generally speaking, a descent method starts with an initial point $u_0$ and, with information of first order, the algorithm finds directions that lead us to the minimum of the objective functional. Also, the method must find the length of the step, $\alpha_r $, along the chosen direction, $w_r $. The basic idea consists in finding $\alpha_r$ and $w_r$ such that:
\begin{equation*}
J(u_r+\alpha_r w_r)< J(u_r), \text{   for     }  \alpha_r>0 
\end{equation*}
in every iteration of the method. 

In the gradient method, for instance, the search direction $w_r$ is determined by
\[
w_r=- \nabla J(u_r).
\]
On the other hand, for the preconditioned descent algorithm proposed in \cite{Gonzalez1}, the search direction $w_r$  is determined by solving the following variational equation 
\[
P_r(w_r,v)=- \nabla J(u_r)^{\top} v_r,  \,\, \forall v_r \in V_h^0,
\]
where the form $P_r:V_h^0 \times V_h^0 \rightarrow \mathbb{R}$ is chosen as a variational approximation of the $p$-Laplacian. 

The algorithm reads as follows
\begin{algorithm}[H]
\caption{General Preconditioned descent algorithm}\label{alg:PCDA}
\begin{algorithmic}[1]
\State Initialize $u_0 \in V_h^0 $ and set $r=0$.
If {$\nabla J(u_r)=0$}, STOP. Otherwise:
\State Find a descent direction $w_r $ by solving the following equation
\begin{equation*}
P_r(w_r,v_r)=- \nabla J(u_r)^{\top}v_r,  \,\,\, \forall v_r \in V_h^0,
\end{equation*}
\hbox{          } if {$1<p<2$}, 
\[
P_r(w_r,v_r)=\int_{\Omega_h} (\epsilon + |\nabla u_r|)^{p-2} \nabla w_r \nabla v_r \,\,dx, \,\, \forall v_r \in V_h^0,
\]
\hbox{          } else if {$p\geq 2$}, 
\[
P_r(w_r,v_r)=\int_{\Omega_h}  \nabla w_r\nabla v_r \,\, dx,  \,\, \forall v_r \in V_h^0,
\]
\hbox{          } end.

\State Perform a line search algorithm to obtain $\alpha_r$.
\State Update $u_{r+1}:=u_r + \alpha_r w_r$ and set $r=r+1$.
\end{algorithmic}
\end{algorithm}
The global convergence and the stability of this Algorithm is guaranteed both in finite and infinite dimension settings. For a deeper analysis of this algorithm, we refer the reader to \cite{Gonzalez1}.
 
\subsection{Line search technique} \label{sec:line_search}
In this section, we describe the line search algorithm which will be used in the implementation of Algorithm \ref{alg:MGOPT_alg}. This algorithm uses polynomial models of the objective functional for backtracking, and it was originally proposed in \cite[Sec. 6.3.2]{Dennis}. 

The algorithm reads as follows

\begin{algorithm}[H]\label{algo:armijo}
\caption{Line search algorithm by polynomial models}
Let $\sigma_1\in (0,\frac{1}{2})$ and set $\alpha_0=1$. 
\begin{algorithmic} [1] 
\item Decide wheter $J(u_r + \alpha_r)> J (u_r) + \sigma_1 \alpha_r \nabla  J(u_r)^\top w_r$ holds. If so, STOP and set $\alpha_r=\alpha_0$. If not:
\item Decide whether the step length is too small. If so, STOP and terminate algorithm: routine failed to locate satisfactory $x_{r+1}$ sufficiently distinct from $x_r$. If not:
\item Decrease $\alpha$ by a factor between 0.1 and 0.5 as follows:
\item On the first backtrack: set $\alpha_r:=\tilde{\alpha}_2= \textrm{argmin} \,m_2(\alpha)$, but constrain the new $\alpha_r$ to be $\geq 0.1$.
\item On all the subsequent backtracks: set $\alpha_r:= \tilde{\alpha}_3 =\textrm{argmin}\, m_3(\alpha)$, but constrain the new $\alpha_r$ to be in $[0.1\alpha_p\,,\, 0.5 \alpha_p]$.
\item Return to step 1.
\end{algorithmic}
\end{algorithm}
If we set
\[
\varphi_r(\alpha):=J(u_r+\alpha w_r),
\]
then the quadratic model $m_2$ is given by
\[
m_2(\alpha):= (\varphi_r(1)-\varphi_r(0)-\varphi_r'(0))\alpha^2 + \varphi_r'(0) \alpha + \varphi_r(0),
\]
while the cubic model $m_3$ is given by
\[
m_3(\alpha):= c \alpha^3 + d\alpha^2+\varphi_r'(0) \alpha + \varphi_r(0),
\]
where
\begin{equation*}
\binom{c}{d}=\frac{1}{\alpha_p - \alpha_{2p}} \left( \begin{array}{cc}
\frac{1}{ \alpha_{p}^2} &  \frac{-1}{ \alpha_{2p}^2}  \\
\frac{- \alpha_{2p}}{\alpha_p^2} & \frac{\alpha_p}{\alpha_{2p}^2}   \end{array}  \right) \left( \begin{array}{c}
 \varphi_r(\alpha_p)-\alpha_r(0)-\alpha_r'(0) \alpha_p\\
\alpha_r(\alpha_{2p}) - \alpha_r(0)- \alpha_r'(0) \alpha_{2p}.  \end{array}  \right)
\end{equation*}
and $\alpha_p$ and $\alpha_{2p}$ are the last two previous values of $\alpha_k$. For further details and examples see \cite[pp. 126-129]{Dennis} and the references therein.

\section{Applications to the Numerical Simulation of Viscoplastic Flow}\label{sec:flow}
In this section, we discuss the application of the MG/OPT algorithm to the numerical simulation of the steady flow of viscoplastic fluids. These materials are characterized by the existence of a yield stress \cite{Chhabra,Duvant-lions,Huilgol,Huilgol1}. This implies that the viscoplastic material exhibits no deformation if the shear stress imposed does not exceed the yield stress, i.e., it behaves like an ideal rigid solid. If the shear stress overpasses the yield stress, the material will deform like a nonlinear viscous fluid in most of the cases. In fact, Herschel - Bulkley and Casson fluids present a nonlinear stress-shear rate, while Bingham fluids behave like a viscous fluid with linear stress-shear rate (see Figure \ref{fig:vp_models}). Summarizing, the existence of the yield stress causes rigid zones and yielded zones in the flow.
\begin{figure}[H] 
\begin{center}
\includegraphics[scale=0.4]{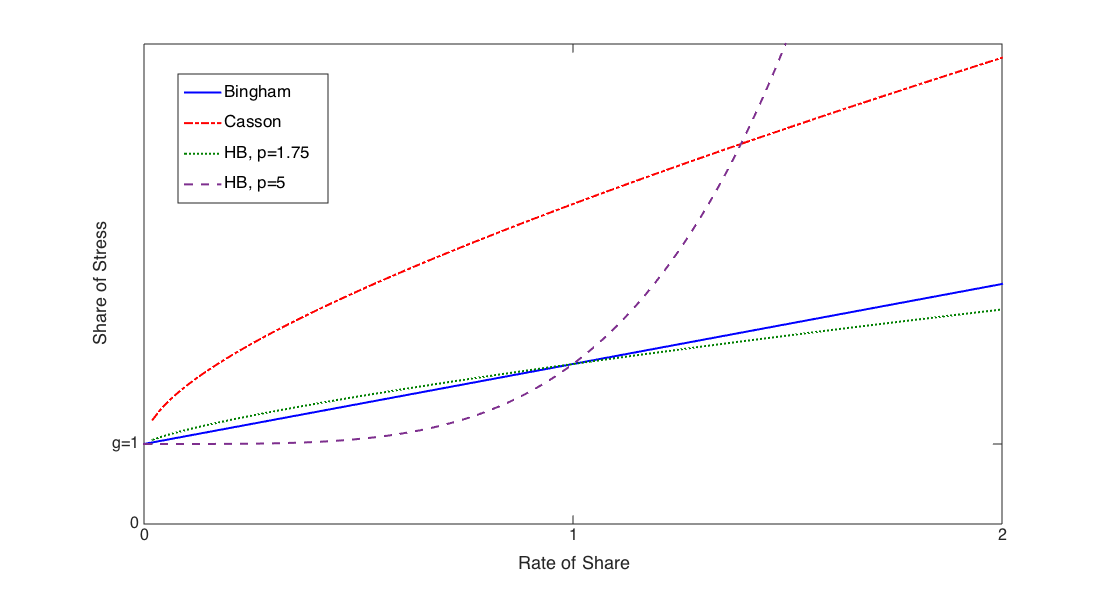}
\centering
\caption{Viscoplastic models}\label{fig:vp_models}
\end{center}
\end{figure}

In this work we focus on the stationary and laminar flow of a viscoplastic fluid in a cylindrical pipe under the effect of a drop in pressure. We consider three classic models for these fluids: Herschel-Bulkley, Bingham and Casson. This kind of flow is a simplified problem in which we can assume that all velocity fields only have a non-zero component in the axial direction. Therefore, by assuming that the velocity fields vanish on the boundary of the pipe (the so called adhesion condition) and that $f$ represents the constant pressure drop, it is well known that the velocity field across the cross-section of the pipe solves the following optimization problem.
\begin{equation}\label{eq:final}
\underset{u_h\in V_h^0} {\min} J_h(u_h):= \phi(\nabla u_h) + \int_{\Omega_h}\psi_\gamma(\nabla u_h)\,dx - \int_{\Omega_h} f u_h\,dx,
\end{equation}
where
\begin{equation*}
\phi(\nabla u_h) = \begin{cases} \begin{array}{ll}
       \displaystyle\frac{1}{p} \int_{\Omega_h}|\nabla u_h|^{p}\,dx, & \text{for Herschel-Bulkley model} \\
       \displaystyle\frac{1}{2}\displaystyle \int_{\Omega_h}|\nabla u_h|^{2}\,dx, & \text{for Bingham model} \\  
       \displaystyle \frac{1}{2}\displaystyle \int_{\Omega_h}|\nabla u_h|^{2}\,dx + \frac{4}{3} \sqrt{g} \int_{\Omega_h}|\nabla u_h|^{\frac{3}{2}}\,dx, & \text{for Casson model.} \\  
        \end{array}
        \end{cases}
\end{equation*}
This variational formulation is motivated by the necessity of representing the free surface that separates the regions in which the material has yielded from those in which it behaves like a rigid solid. Regarding more general problems, like the flow in 2D and 3D geometries, the variational formulation also leads to optimization problems formulated in divergence-free spaces. This fact suggests that the present methodology can be generalized to these more challenging problems. We refer the reader to \cite{Duvant-lions,Gonzalez1,Huilgol} for a more detailed explanation of the variational approach to the flow problems.

In order to validate our results, we introduce the theoretical velocity distribution, in cylindrical coordinates, for a circular pipe flow. Here $r=r_0=2g$ (See \cite{Huilgol}). 
\begin{itemize}
\item Herschel-Bulkley
\end{itemize}
\begin{equation}\label{analityc_sol_HB}
u(r) = \begin{cases} \begin{array}{ll}
       \displaystyle\frac{(1-r_0)^{1+\beta}}{2^{\beta}(1-\beta)} & 0\leq r \leq r_0 \\\\
       
       \displaystyle\frac{\left( (1-r_0)^{1+\beta}-(r-r_0)^{1+\beta}\right)}{2^{\beta}(1+\beta)}, &  r_0\leq r \leq 1\\ 
        \end{array}
        \end{cases}
\end{equation}
with $\beta=\frac{1}{p-1}$.
\begin{itemize}
\item Bingham
\end{itemize}
\begin{equation}\label{analityc_sol_BH}
u(r) = \begin{cases} \begin{array}{ll}
       \displaystyle\frac{1}{4}(1-r_0)^{2} & 0\leq r \leq r_0 \\\\
       \displaystyle\frac{1}{4}\left( (1-r_0)^{2}-(r-r_0)^{2}\right), &  r_0\leq r \leq 1\\ 
        \end{array}
        \end{cases}
\end{equation}
\begin{itemize}
\item Casson
\end{itemize}
\begin{equation}\label{analityc_sol_CSS}
u(r) = \begin{cases} \begin{array}{ll}
       \displaystyle\frac{1}{12}(3-8r_0^{1/2}+6r_0-r_0^{2}) & 0\leq r \leq r_0 \\
\\
       \displaystyle\frac{1}{4}(1-r^2)-\frac{2}{3}r_0^{1/2}(1-r^{3/2})+\frac{1}{2}r_0(1-r), &  r_0\leq r \leq 1\\ 
        \end{array}
        \end{cases}
\end{equation}

Hereafter, we discuss the performance of the MG/OPT algorithm when applied to the numerical solution of \eqref{eq:final}. All the numerical experiments in this paper are implemented in MATLAB (R2015a) and run on an Intel Core i5 processor with 2.5 GHz. The MG/OPT algorithm was implemented in a V-cycle scheme with $\nu_1$ and $\nu_2$ as the pre and post optimization iterations. In all experiments the right hand side was set $f=1$. We initialize the algorithms with the solution of the Poisson problem $-\Delta u_h=f$. We performed the numerical  experiments in the unit circle domain. Also, we compare the performance of the MG/OPT algorithm with the performance of the preconditioned descent algorithm when solving the same problem in the finest grid. 

Regarding the stopping criteria, we stop the algorithms \ref{alg:MGOPT_alg} and \ref{alg:PCDA} as soon as the expression $\|\nabla J_{k}(u_k)\|$ is reduced by a factor of $10^{-7}$. 

\subsection{Herschel-Bulkley}
 
The Herschel-Bulkley model can be seen as a power-law model with plasticity \cite{Gonzalez1, Huilgol}. The introduction of the parameter $p$, known as the flow index, allows the model to represent the behaviour of several viscoplastic fluids. In fact, if $1<p<2$, the material exhibits a pseudoplastic or shear-thinning behaviour. On the other hand, if $ p>2$ the fluid behaves as a shear-thickening material (see Figure \ref{fig:vp_models}). This versatility makes the Herschel-Bulkley model to be widely used to simulate several materials including liquid foams, whipped cream, fluid foods, silly putty and some polymers \cite{Chhabra}. 

\subsubsection*{Experiment 1 (case $1 < p < 2$)}  

In this experiment we compute the solution for problem \eqref{eq:final} with $p=1.75$ and different values of $g$, ranging from $g=0$ to $g=0.4$. We fix the preconditioned descent algorithm (see \cite{Gonzalez1}) as the underlying optimization algorithm $S_{opt}$, with set $\epsilon=10^{-6}$. MG/OPT V-cycles are carried out in 5 grids, with $8321$ nodes in the finest grid and $41$ nodes in the coarsest one. 

In Table \ref{table:exp1_hb}, we compare the MG/OPT algorithm and the preconditioned descent algorithm. Following \cite{Nash}, we stop the MG/OPT Algorithm \ref{alg:MGOPT_alg} at a given number of iterations and compare the numerical performance with the preconditioned descent algorithm. In particular, we compare the error $Err_s$, which is defined as the difference between the current solution and a highly accurate solution, measured in the discrete $L_2$-norm. The accurate solution is obtained by running the globally convergent preconditioned descent algorithm on the finest mesh.  We report $Err_s$ in the third and seventh columns for MG/OPT and the preconditioned algorithm, respectively. 
Also, since we know the analytical solution for the velocity distribution along the radio in the circular pipe \eqref{analityc_sol_HB}, in the fifth and eight columns of Table \ref{table:exp1_hb}  we present the error for the  constant plug flow velocity , i.e., $ Err_{pf}=|u_{\text{a}} - u_{\text{num}}|$, where $u_{\text{a}}$ is the analytical plug flow velocity calculated with \eqref{analityc_sol_HB} and $u_{\text{num}}$ stands for the plug flow velocity calculated numerically with both algorithms.  With $g=0$, MG/OPT performs less iterations than the preconditioned descent algorithm. However, since at each iteration or V-cycle of MG/OPT the algorithm performs $\nu_1+\nu_2=4$ optimization steps at the finest grid, the total number of optimization steps in the finest grid is not reduced substancially in comparison with the preconditioned descent algorithm. In the subsequent cases with a larger $g$, MG/OPT iterations are about one seventh of the preconditioned descent algorithm iterations. Also, in comparison with the underlying optimization algorithm, MG/OPT reduces by almost half the number of optimization steps in the finest grid at every stage of the computation, with $g=0.2$. This behaviour is replicated in the advanced stages of MG/OPT with $g=0.4$. 

\begin{table}[H]
\begin{center}
\begin{tabular}{c l l l l |l l l l l}
  \hline
        \multicolumn{5}{c|}{MG/OPT ( $\nu_1=\nu_2=2$)} 
          & \multicolumn{4}{|c}{Preconditioned descent algorithm} \\ \hline            
 g&  It. & $Err_{s}$ & Plug Flow vel.& $Err_{pf}$ &  It. & $Err_{s}$ & Plug Flow vel. & $Err_{pf}$  \\  \hline
     &1 &9.69e-09&0.169& 1.42e-04  &5&3.36e-09& 0.1700&2.61e-05\\    
  0  & 2&2.44e-14& 0.1700& 3.57e-05&9&4.97e-14& 0.1700&3.57e-05\\ 
    &  3&1.08e-18& 0.1700& 3.59e-05&12&1.0e-18& 0.1700&3.59e-05\\ \hline
    &1&4.29e-05& 5.91e-02& 7.5e-03&6&3.09e-05& 5.94e-02&7.7e-03\\    
    & 3& 3.94e-06& 5.36e-02& 1.9e-03&31&3.99e-06& 5.35e-02&1.9e-03\\ 
  0.2  & 5&2.82e-06& 5.33e-02& 1.7e-03&40&2.78e-06& 5.33e-02&1.6e-03\\ 
    & 7&1.72e-06& 5.30e-02& 1.3e-03&51&1.78e-06& 5.29e-02&1.3e-03\\ 
    &  9& 1.00e-06& 5.27e-02& 1.1e-03&61&1.06e-06& 5.27e-02&1.0e-03\\ \hline
    &1& 1.6e-04& 1.76e-02& 1.3e-02&5&1.13e-04& 1.5e-02&1.1e-02\\   
    & 3& 1.08e-05& 6.8e-03& 2.0e-03&16&1.02e-05& 6.7e-03&2.7e-03\\ 
  0.4  & 5&3.80e-06& 5.7e-03& 1.7e-03&35&3.73e-06& 5.7e-03&1.7e-03\\ 
    & 7&1.86e-06& 5.3e-03& 1.3e-03&49&1.81e-06& 5.0e-03&1.2e-03\\ 
    &  9& 1.00e-06& 5.0e-03& 1.0e-03&58&1.09e-06& 5.0e-03&1.0e-03\\ \hline
\end{tabular}
\caption{Results of the resolution of problem \eqref{eq:final} with $p=1.75$, $g=0.2$ ,$ \gamma=10^3 $ and $f=1$.} \label{table:exp1_hb}
\end{center}
\end{table}

The resulting velocity fields $u_h$ are displayed in Figure \ref{fig-ex1_vel}.  The regions with constant velocities in the center of the pipe corresponds to the plug flow velocity, where the material behaves like a rigid body. These regions, of course, do not appear with $g=0$.
\begin{figure}[H]
\begin{minipage}[htb]{0.5\textwidth}
\begin{center}
\includegraphics[scale=0.4]{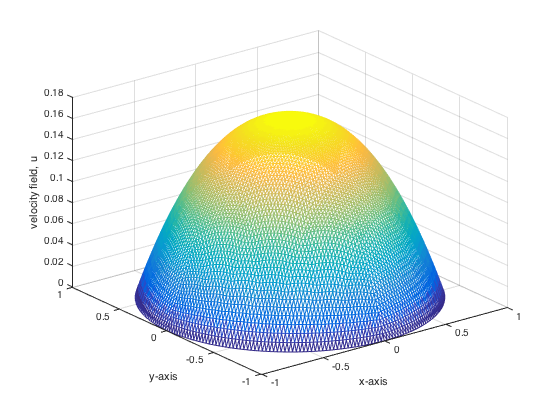}
\end{center}
\end{minipage}\hspace{0.5cm}
\begin{minipage}[htb]{0.5\textwidth}
\begin{center}
\includegraphics[scale=0.4]{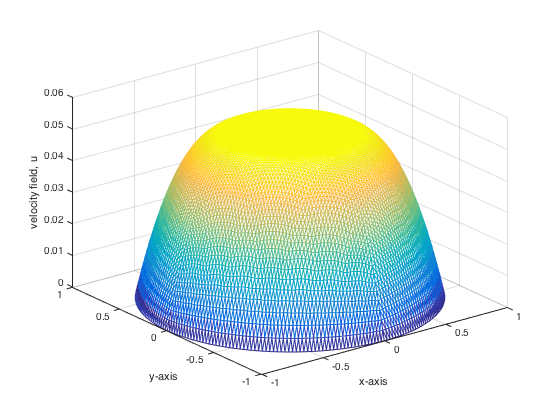}
\end{center}
\end{minipage}\hspace{0.5cm}
\begin{minipage}[htb]{\textwidth}
\begin{center}
\includegraphics[scale=0.4]{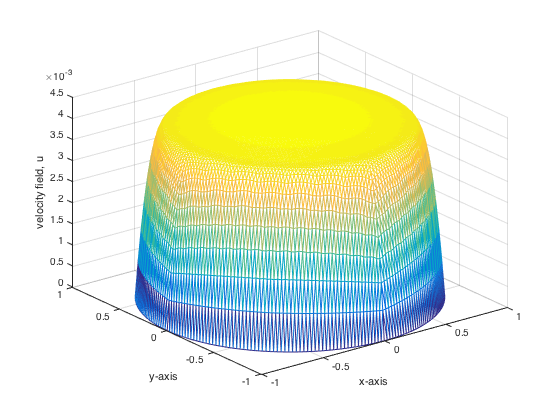}
\end{center}
\end{minipage}
\caption{Calculated velocity $u_h$ for the Herschel-Bulkley model with $p=1.75$. Parameters: $g=0$ (top left),  $g=0.2$ (top right) and $g=0.4$ (bottom), $\gamma=10^3$.}\label{fig-ex1_vel}
\end{figure}

\subsection{Bingham}
Bingham fluids are viscoplastic materials that can be seen as a particular case of the Herschel-Bulkley model when $p=2$. The main characteristic of Bingham fluids is that when the shear stress exceeds the yield stress, the material exhibits a close to linear stress-strain relation (see Figure \ref{fig:vp_models}).

\subsubsection*{Experiment 2} 
The following experiment was carried out in a disk domain with the same characteristics as the previous experiments. We consider the following parameters: $g=0.4$ and $ \gamma=10^3$. In this experiment, we execute a comparison in the same terms as in the previous experiment. In Table \ref{table:bingham_ex11}, MG/OPT with $\nu_1=\nu_2=2$ does not reduce the optimization steps in the finest grid (each iteration of MG/OPT performs $4$ optimization steps in the finest grid) and $Err_{s}$  hardly decreases. Due to this unexpected behaviour, the number of optimization steps in the post-optimization procedure was increased  and decreased in the pre-optimization procedure by setting $\nu_1=1$ and $\nu_2=3$. This  new computation is tabulated In Table \ref{table:bingham_ex12}. In this case, the optimization steps in the finest grid are reduced substantially. The numerical experience in this experiment tells us that the restriction procedure introduce errors to the solution at each level. Then, these errors are reduced more efficiently by increasing the number of optimization steps in the post optimization phase. 
\begin{table}[H]
\begin{center}
\begin{tabular}{c l l l l |l l l l l}
  \hline
        \multicolumn{5}{c|}{MG/OPT ( $\nu_1=\nu_2=2$)} 
          & \multicolumn{4}{|c}{Preconditioned descent algorithm} \\ \hline            
 g&  It. & $Err_{s}$  & Plug Flow vel.&$Err_{pf}$  &  It. & $Err_{s}$  & Plug Flow vel. &$Err_{pf}$  \\  \hline
    &1& 1.0e-03& 4.6e-02& 3.63e-02&4&1.1e-03& 5.36e-02&4.36e-02\\   
    & 3& 7.38e-05& 1.92e-02& 9.2e-03&12&6.74e-05& 2.01e-02&1.01e-02\\ 
  0.4  & 5& 4.47e-05& 1.70e-02& 7.0e-03&15&4.08e-05& 1.74e-02&7.4e-03\\ 
    & 7& 2.96e-05& 1.61e-02& 6.1e-03&25&2.93e-05& 1.62e-02&6.2e-03\\ 
    &  9& 2.59e-05& 1.58e-02& 5.8e-03&33&2.51e-05& 1.57e-02&5.7e-03\\ \hline
\end{tabular}
\caption{Results of the resolution of problem \eqref{eq:final} with $p=2$, $ \gamma=10^3 $ and $f=1$.}\label{table:bingham_ex11}
\end{center}
\end{table}
\begin{tabular}{c l l l l |l l l l l}
  \hline
        \multicolumn{5}{c|}{MG/OPT ( $\nu_1=1,\nu_2=3$)} 
          & \multicolumn{4}{|c}{Preconditioned descent algorithm} \\ \hline            
 g&  It. & $Err_{s}$  & Plug Flow vel.& $Err_{pf}$  &  It. & $Err_{s}$  & Plug Flow vel. &$Err_{pf}$  \\  \hline
    &1& 4.95e-05& 2.3e-02& 1.30e-02&13& 4.80e-05& 1.81e-02& 8.1e-03\\   
    & 3& 2.31e-06& 1.32e-02& 3.2e-03&215&2.32e-06& 1.30e-02&3.0e-03\\ 
  0.4  & 5& 1.29e-06& 1.27e-02& 2.7e-03&244&1.28e-06& 1.27e-02&2.7e-03\\ 
    & 7& 8.51e-07& 1.26e-02& 2.6e-03&260&8.59e-07& 1.25e-02&2.5e-03\\ 
    &  9& 5.75e-07& 1.24e-02& 2.4e-03&273&5.83e-07& 1.24e-02&2.4e-03\\ \hline
\end{tabular}

\begin{table}[H]
\begin{center}

\caption{Results of the solution of problem \eqref{eq:final} with $p=2$, $ \gamma=10^3 $ and $f=1$.} \label{table:bingham_ex12}
\end{center}
\end{table}

 The resulting velocity field is displayed in Figure \ref{fig-ex2_bg_vel}. Since the yield stress $g=0.4$ is high, the plug flow covers a large part of the cross section of the pipe.
 \begin{figure}[H] 
\begin{minipage}[htb]{\textwidth}
\begin{center}
\includegraphics[scale=0.5]{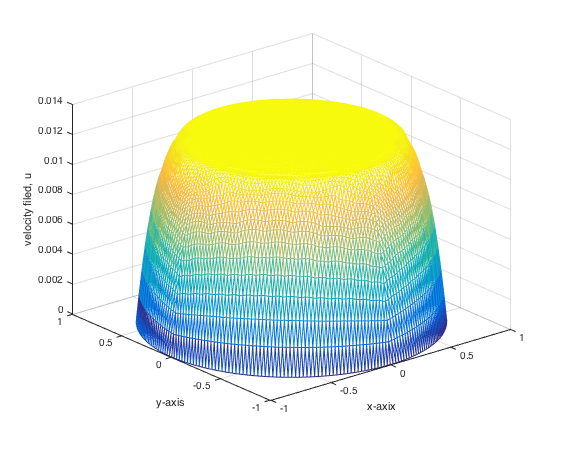}
\caption{\footnotesize{Calculated velocity $u_h$ for Bingham model. Parameters: $ g=0.4$ and $\gamma= 10^3$}}
\label{fig-ex2_bg_vel}
\end{center}
\end{minipage}
\end{figure}

\subsection{Herschel-Bulkley: case $ p > 2$}

We now analyze the behaviour of the MG/OPT algorithm in the case $p > 2$, which usually gives rise to instabilities in the performance of the numerical algorithms (see \cite{BarretLiu,Bermejo,Glowinski,Huang}). 
\subsubsection*{Experiment 3}
In the following experiment we present the performance of the MG/OPT algorithm for Herschel-Bulkley with $p=5$ and $g=0.1$. When $p$ increases, the optimization problem \eqref{eq:final} is difficult to solve \cite{Huang}. Additionally, since we restrict the approximate solution to coarser grids, at the coarsest mesh the restriction procedure fails in approximating efficiently the plug flow. Hence, in this case we only consider MG/OPT in a V-cycle with 3 grids instead of 5 grids (in order to avoid several restriction procedures).  Therefore, we work with $2113$ nodes in the finest grid and $145$ nodes in the coarsest one. The preconditioned descent algorithm \cite{Gonzalez1} is fixed as the underlying optimization algorithm $S_{opt}$. With the previous setting, in Table \ref{table:exp_hb_2} the MG/OPT algorithm (with $\nu_1=\nu_2=2$) does not  reduce the optimization steps in the finest grid  in the first two cycles. However, in the third cycle  MG/OPT performs one quarter of the optimization steps of the preconditioned descent algorithm in the finest grid.
\begin{table}[H]
\begin{center}
\begin{tabular}{c l l l l |l l l l l}
  \hline
        \multicolumn{5}{c|}{MG/OPT ( $\nu_1=\nu_2=2$)} 
          & \multicolumn{4}{|c}{Preconditioned descent algorithm} \\ \hline            
 g&  It. & $Err_{s}$  & Plug Flow vel.& $Err_{pf}$  &  It. & $Err_{s}$  & Plug Flow vel. &$Err_{pf}$  \\  \hline
    &1& 6.67e-05& 4.92e-01& 1.63e-02&4& 6.91e-05& 4.91e-01& 1.69e-02\\   
  0.1  & 2& 6.38e-07& 5.03e-01& 5.6e-03& 9&2.78e-07& 5.01e-01&7.1e-03\\ 
    & 3& 9.31e-09& 5.02e-01& 6.7e-03&51&8.98e-09& 5.01e-01&7.6e-03\\ \hline
\end{tabular}
\caption{Results of the resolution of problem \eqref{eq:final} with $p=5$, $ \gamma=10^3 $ and $f=1$.} \label{table:exp_hb_2}
\end{center}
\end{table}
The resulting velocity field is displayed in Figure \ref{fig-hb_2}. Now we  are in the case of a shear-thickening material. Since the shear stress transmitted by a fluid layer decreases toward the center of the pipe, the velocity takes a conical form. 

\begin{figure}[H] 
\begin{minipage}[htb]{\textwidth}
\begin{center}
\includegraphics[scale=0.5]{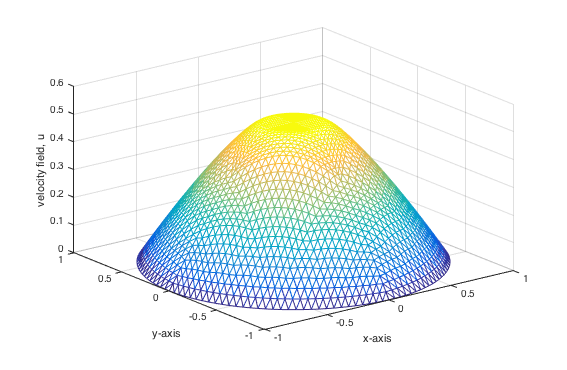}
\caption{\footnotesize{Calculated velocity $u_h$ for Herschel-Bulkley model for $p=5$. Parameters: $ g=0.1$ and  $\gamma= 10^3$}}
\label{fig-hb_2}
\end{center}
\end{minipage}
\end{figure}

\subsection{Casson}
The Casson model is a viscoplastic model that was first developed for modeling printing inks. It has also been used to model food flow behaviour such as chocolate and cocoa products \cite{Rao}, and has been applied to biorheology models like hemodynamics and viscometric flows \cite{Walawender}. The Casson model turns out to be special since it involves the sum of the terms $\frac{1}{2}\displaystyle \int_{\Omega_h}|\nabla u_h|^{2}\,dx + \frac{4}{3} \sqrt{g} \int_{\Omega_h}|\nabla u_h|^{\frac{3}{2}}\,dx$. The first term corresponds to the Bingham structure and the second one to the Herschel-Bulkley structure with $p=\frac{4}{3}<2$. Thus, we decide to use the preconditioned descent algorithm with $\epsilon=10^{-6}$. 

\subsubsection*{Experiment 4}
In the following experiment we present the performance of the MG/OPT algorithm for the Casson problem with parameter $g=0.2$. In this experiment we show the behaviour of MG/OPT with a full multigrid (FMG) scheme \cite{Brandt}. Here, our aim is to show the time reduction when solving the Casson problem with MG/OPT and without it. We compare this scheme with a different number of grids. In Table \ref{table:exp4_casson}, the first row corresponds to the solution computed by the preconditioned descent algorithm in a circle domain in a fine grid with $8321$ nodes. We take as reference the number of iterations and the execution time of this algorithm i.e., the CPU time until the error $Err_{s}$  reaches a tolerance of $1e-07$. The subsequent rows correspond to the solution computed by the MG/OPT $(\nu_1=\nu_2=2)$ algorithm. For each row we present the number of grids used and the number of nodes in the coarsest grid. At each row, a new grid is added to the scheme.

Since we know the analytical solution \eqref{analityc_sol_CSS} in the circular pipe, we present the error $Err_{pf}$ for the plug flow velocity in the fifth column. FMG  performs $1$ V-cycle with $\nu_1=\nu_2=2$ optimization steps in the finest grid, totally we have 4 finest grid optimization steps whereas the preconditioned descent algorithm performed 10 iterations in order to achieve the same error. Let us notice that, as more grids are added to the scheme, we achieve higher CPU time savings in all cases. With 5 grids the time reduction achieved was around 40\%.

\begin{table}[H]
\begin{center}
\begin{tabular}{c c c @{\extracolsep{0.3 cm}}  c c c }
\hline
No.Grids& No. Nodes& It & Plug flow & $Err_{pf}$  & Rel CPU Time \\ \hline
1& 8321 &10 &1.55e-02&  5.03e-04&1 \\
2& 2113 &2 &1.51e-02 & 4.00e-05 & 0.93 \\
3& 545  &3 &1.51e-02 & 2.95e-05 &0.75\\
4& 145  &4 &1.51e-02 & 4.84e-05 &0.70\\
5& 41   &5&1.50e-02& 5.99e-06 &0.69\\ \hline
\end{tabular}
\caption{Results of the resolution of problem \eqref{eq:final} with $ \gamma=10^3 $ and $f=1$.} \label{table:exp4_casson}
\end{center}
\end{table}

\begin{figure}[H] 
\begin{minipage}[htb]{\textwidth}
\begin{center}
\includegraphics[scale=0.4]{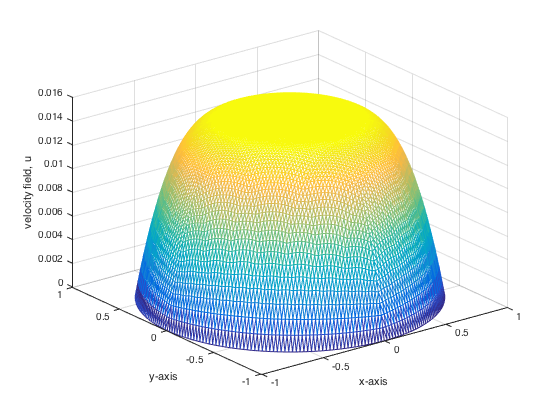}
\caption{\footnotesize{Calculated velocity $u_h$ for  Casson model. Parameters: $ g=0.2, \gamma= 10^3$}}
\label{fig-bing}
\end{center}
\end{minipage}
\end{figure}

\noindent \textrm{\small{The MATLAB codes are available in \tt www.modemat.epn.edu.ec/$\sim$sgonzalez/Publications.html.}}

\section{Conclusions and Outlook}\label{sec:conclusions}
We  proposed and analyzed a multigrid for optimization (MG/OPT) algorithm for the numerical solution of a class of quasilinear variational inequalities of the second kind. We analyzed the variational inequality via the minimization of the associated energy functional.  First, we regularized the non-differentiable part of the energy by using a Huber regularization approach. Next, we proposed a finite element discretization for the problem, and we analyzed the differentiability of the discretized functional. In particular, we proved that the Jacobian of the discretized functional is slantly differentiable. The MG/OPT algorithm was presented and all of the involved transfer operators analyzed. The convergence of the MG/OPT algorithm was established by using the mean value theorem for slantly differentiable functions and the global convergence of the underlying optimization algorithm.  The main issues regarding the implementation of the algorithm were explained, and we described the type of global convergent deepest descent methods used as underlying optimization algorithms. We showed that several classical models for viscoplastic flow correspond to the class of variational inequalities under study. Therefore, we focussed the numerical experiments on this kind of problems. Particularly, we computed the solution for the Herschel-Bulkley, Bingham and Casson models. In all the experiments presented, we observed CPU-time savings, especially when working on the finest meshes. This showed that the MG/OPT algorithm is indeed an efficient tool for dealing with large scale problems.

In order to continue this research, we consider that the study of a more general class of variational inequalities is an interesting perspective. Since the functional $J_{\gamma,h}$ has a second slant derivative, the development of generalized Newton-type methods as smoothers for the FAS scheme looks like a promising field of research. Finally, the analysis of more challenging problems, such as the $p$-Stokes problem, problems coming from glaceology, geophysics and hemodynamics is of great interest. 

\section*{Acknowledgements} We thank Prof. Juan Carlos De los Reyes (MODEMAT-EPN) for his keen observations in the theory of semismooth functions. We also thank the anonymous referees for many helpful comments which lead to a significant improvement of the article.

\end{document}